\documentclass{cimart}

\usepackage{amsthm,array,mathrsfs,amscd}
\usepackage{amsmath,amssymb,enumitem}
\usepackage{dsfont}

\usepackage{tikz}
\usepackage{tikz-cd}
\usetikzlibrary{patterns}
\usetikzlibrary{calc}
\usetikzlibrary{decorations.markings}

\def\Field{\mathbf{K}}

\def\Sspan#1{\langle #1\rangle}
\def\Span#1{\mathrm{span}(#1)}
\def\Alg{\mathbb{A}}
\def\congp{\stackrel{\centerdot}{\cong}}

\DeclareMathOperator{\type}{type}

\DeclareMathOperator{\im}{im}

\DeclareMathOperator{\Idm}{Idm}

\DeclareMathOperator{\Aut}{Aut}
\DeclareMathOperator{\End}{End}

\title{
    Inner isotopes associated with automorphisms of commutative associative algebras
    }

\author{
    Vladimir Tkachev
    }

\authorinfo[
    Vladimir Tkachev]{
    Link\"oping University, Sweden}{%
    vladimir.tkatjev@liu.se
    }

\abstract{%
   The principal observation of the present paper is that an inner isotopy (i.e. a principal isotopy defined by an algebra endomorphism) is a very helpful instrument in constructing and studying interesting classes of nonassociative algebras.
By using methods developed in the paper, we define a new class of commutative nonassociative algebras obtained by inner isotopy from commutative associative polynomial algebras.
There is a natural bijection between isomorphism classes of our algebras and integer partitions of the algebra dimensions.
Among the interesting features of the nonassociative algebras constructed are that these algebras are generic, some of examples are axial and metrized algebras. We completely describe both the set of algebra idempotents and their spectra.
    }



\VOLUME{32}
\YEAR{2024}
\ISSUE{2}
\firstpage{153}
\DOI{https://doi.org/10.46298/cm.12223}

\begin{document}

\bigskip

\begin{small}
\begin{quotation}
\hfill \textsl{To the memory of  Yakov Krasnov (1956-2023)}\par

\hfill\textsl{ my friend and colleague.}
\end{quotation}
\end{small}

\section{Introduction}\label{sec:intro}
By an algebra $(\Alg, \ast)$ we understand a nonassociative algebra over a  field $\Field$ of $\mathrm{char}(\Field) \ne 2,3$ with multiplication $\ast$. An element $c\in \Alg$ is called an idempotent if $c\ast c=c$. The set of nonzero idempotents of $(\Alg,\cdot)$ is denoted by $\Idm(\Alg,\ast)$. Given an element $a\in (\Alg,\ast)$, we denote by  $L_{\ast}(a)x \to a\ast x$ the left multiplication operator by $a$ and $\Alg_{\lambda}(a)$ the kernel of $(\lambda \mathds{1}_\Alg-L_{\ast}(a))$, where $\mathds{1}_\Alg$ is the identity operator.

A commutative algebra $(\Alg,\ast)$ is called \textit{isospectral} if the spectrum of $L_{\ast}(c)$ is the same for any nonzero idempotent $c$.
By using the syzygy method, it was established in \cite{KrTk18a}  that if an isospectral algebra $(\Alg, \ast)$ is generic (see Definition~\ref{def:gener} below) then the common spectrum of the algebra idempotents consists of simple eigenvalues satisfying $\lambda^{\dim \Alg}=1$. A further analysis given in \cite{KrTk23a} reveals that under some mild conditions, an isospectral generic algebra must be \textit{medial}, i.e. the algebra multiplication associates on pairs:
$$
(x\ast y)\ast(z\ast t)=(x\ast z)\ast(y\ast t),
$$
and, moreover, such an algebra is an inner isotopy of  a certain commutative associative algebra.
Recall that given an algebra $(\Alg, \ast)$, its \textit{inner isotope} $(\Alg, \ast_h)$ is the vector space $\Alg$ with the new multiplication
$$
x\ast_h y=h(x\ast y)=h(x)\ast h(y),
$$
where $h\in \Aut(\Alg, \ast)$ is an automorphism of $(\Alg, \ast)$. More precisely, the medial isospectral generic algebras discussed in \cite{KrTk23a} are exactly the inner isotopes of the quotient polynomial algebra $\Field[z]/(z^n-1)$ under the automorphism $\tau\in \Aut(\Field^n,\bullet)$ acting by  substitution $[p(z)]\to [p(\epsilon_n z)]$, where $\epsilon_n$ is a primitive root of unity of order $n$.
The corresponding inner isotope algebra $(\Field[z]/(z^n-1),\bullet_\tau)$ has many distinguished properties (see Section~\ref{subsec:3} below for $n=3$ and Section~\ref{sec:tau} for the general case).

In this regard, it is natural to find the full automorphism group $\Aut(\Field^n,\bullet)$ and characterize the corresponding (isomorphy classes of) inner isotopes of $(\Field^n,\bullet)$.
A similar approach is also relevant for an arbitrary algebra $(\Alg,\ast)$ and its inner isotopes $(\Alg,\ast_h)$.

These questions were an original motivation for the present paper.
As we shall see, an inner isotopy is a very helpful instrument for constructing  of interesting classes of nonassociative algebras.
In particular, this approach is already fruitful for the simplest possible case when the initial algebra is commutative and associative. In that case, any nontrivial inner isotopy destroys  the associativity but not so much: any inner isotope of a (commutative) associative algebra is always a medial algebra, see Corollary~\ref{cor:med}.

A relevant  in the present context is the mentioned above concept of \textit{generic}  nonassociative algebras \cite{KrTk18a}. The definition comes back to Segre's observation~\cite{Segre} that idempotents in an $n$-dimensional commutative algebra $(\Field^n, \ast)$ over an algebraically closed field $\Field$ with a basis $\{e_i\}_{1\le i\le n}$ can be interpreted as the set of common zeros of $n$ quadratic polynomials
$$
\Phi_k(x):=(x\ast x-x)_k=\sum_{1\le i,j\le n}a_{ijk}x_ix_j-x_k=0, \qquad 1\le k\le n,
$$
in $\Field^n$, where $x=\sum_{1\le i\le n}x_ie_i$ and $a_{ijk}$ are the structural constant of $\ast$ in the basis $\{e_i\}$. By the B\'{e}zout theorem the number of intersection points (i.e. the algebra idempotents) properly counted in the projective space $\Field \mathbb{P}^n$ is either $2^n$ or infinite. An intersection point $c$ is simple if the quadrics are in relative general position at $c$. On the algebra level, the latter is equivalent to that the idempotent $c$ is \textit{regular} \cite{Walcher1}, i.e. the Jacobian of the quadratic endomorphism $\Phi(x):\Field^n\to \Field^n$ is nonzero: $\det (D\Phi(x))=\det (L_{\ast}(c)-\frac12\mathds{1}_\Alg)\ne0$. Then the B\'{e}zout estimate holds:
\begin{equation}\label{Bezout}
\text{the number of \textit{regular} idempotents of $\Alg$ } \le 2^{\dim \Alg}.
\end{equation}
In general, if the ground field $\Field$ is not algebraically closed then the algebra $(\Alg,\ast,\Field')$ over an algebraic extension $\Field'$ of $\Field$ has the same dimension, and any idempotent regular in $(\Alg,\ast,\Field)$ is a regular idempotent  in $(\Alg,\ast,\Field')$. Applying \eqref{Bezout} to $(\Alg,\ast,\Field')$, this implies that \eqref{Bezout} also holds in $(\Alg,\ast,\Field)$. This motivates the following definition.

\begin{definition}\label{def:gener}
A commutative nonassociative algebra $\Alg$ over an arbitrary field $\Field$ is called \textbf{generic} if it has exactly $2^{\dim \Alg}$ distinct regular idempotents.
\end{definition}

\begin{remark}\label{rem:generic}
An algebra on a vector space $V$ is uniquely identified with a point of the set of all bilinear multiplication $V^*\otimes V^*\otimes V$ on $V$. In this sense, the subset of generic algebras is a Zariski open subset of $V^*\otimes V^*\otimes V$. The above definition has  appeared in \cite{KrTk18a}, \cite{KrTk18b} and it should not be confused with some similar analogues, for example in \cite{tevgen}, \cite{Tevelev}.
\end{remark}

In this paper, we show that under some natural assumptions any inner isotope of a commutative associative algebra is generic and, moreover, we explicitly characterize the set of idempotents  and  their spectral properties.

Another relevant concept here is the so-called axial algebras, i.e. the algebras generated by a finite  subset of idempotents which satisfy a common fusion law. More precisely, a \textbf{fusion law} is a set $\mathcal{F}\subset \Field$ together with
a symmetric binary map $\theta:\mathcal{F}\times \mathcal{F}\to 2^{\mathcal{F}}$.

\begin{definition}[\cite{KhasrawShpectorov20}]\label{def:ax}
Given a fusion law $\mathcal{F}$, a commutative algebra $(\Alg,\ast)$ over $\Field$ together with a distinguished subset of elements $X$ (the called \textit{axes}) is an $\mathcal{F}$-axial algebra if
$(\Alg,\ast)$ is generated by $X$, for each $c \in  X$, $c$ is a semisimple idempotent, namely $\Alg=\oplus_{\lambda\in \mathcal{F}}\Alg_{\lambda}(c)$ and  for $\lambda,\mu\in \mathcal{F}$:
$$
\Alg_{\lambda}(c)\ast \Alg_{\mu}(c)\subset \Alg_{\theta(\lambda, \mu)}(c).
$$
An axial algebra is called \textit{primitive} if each idempotent in $X$ is primitive, i.e. $\dim \Alg_{1}(a)=1$ for any $a\in X$.
\end{definition}

The isospectral generic algebras considered in \cite{KrTk18a} and \cite{KrTk23a} are   \textit{axial} algebras.
The axial algebra concept is an important tool in understanding  of finite groups: it appears that many interesting groups (for example, 3-transposition groups including certain simple sporadic groups, in particular, the monster group) arise as automorphism groups of cubic forms on suitable modules \cite{Smith77}, \cite{Norton94}, \cite{Suzuki86b}, \cite{Ryba96}, \cite{Griess85} by virtue of a correspondence between certain involutions generating a group and a distinguished family of idempotents in an appropriate  (non)associative commutative algebra.
Such a correspondence normally is very individual and drastically depends on a source group/algebra, as well as its combinatorial or geometrical realizations \cite{DeMPS20}, \cite{Fox2022}. Some recent developments in the axial algebra project can be found \cite{Rehren17}, \cite{HRS15}, \cite{DeMedts2018}, \cite{FranchiShpectorov22}, \cite{KhasrawShpectorov20}, \cite{HallShpectorov21}, \cite{CastilloMcInroy21} and the references therein.

The algebras discussed in the present paper fit perfectly this context and provide us with new examples of axial algebras with known automorphism groups. We only outline some partial results in this direction (see especially the explicit examples in Section~\ref{sec:example}), while the general discussion will be addressed elsewhere in the second part of this paper.

\begin{remark}
We are very grateful to the referees for pointing us out the existence of the isomorphism \eqref{isom}. In the original version of our paper  \cite{Tk23pr}, we used a different exposition based on the polynomial model $\Field[z]/P$. The isomorphism \eqref{isom}  simplifies several proofs below and the structure of the idempotent set becomes more transparent  when written in this form. Still, we believe that the polynomial model $\Field[z]/P$ is also of interest, especially for some particular choices of the polynomial $P$; see fore example, the criterion given in Proposition~\ref{pro:Lam} below.  We refer an interested reader to \cite{Tk23pr} for the original presentation and more details concerning the polynomial model.
\end{remark}

\bigskip
The paper is \textbf{organized} as follows. In Section~\ref{sec:prelim}, we recall some general concepts and facts used in the paper. In section~\ref{sec:inner}, we outline the general properties of inner isotopes of an arbitrary algebra and in section~\ref{sec:comm} we specify these results for commutative associative algebras.
In particular, we show that an inner isotope of a commutative associative  algebra must be medial.
In section~\ref{sec:categor}, we develop an appropriate category-theoretical context for our considerations.
The main result of this section states that the categories of  calibrated special commutative medial algebras is isomorphic to the category of calibrated commutative associative algebras.
 In section~\ref{sec:idem} we discuss the properties of idempotents in an arbitrary medial algebra.
In section~\ref{sec:auto} we study the automorphism group of a quotient polynomial algebra and determine its inner isotopes.
We characterize the set of idempotents and theirs spectra in Section~\ref{sec:single}. The automorphism groups of the obtained algebras are studied in Section~\ref{sec:automorph}.
Finally, in section~\ref{sec:example} we illustrate our results  in the case $n=3$.

\section{Preliminaries}\label{sec:prelim}
By $\bar n$ we denote the set of indices $\{1,2,\ldots,n\}$. We recall some standard definitions following \cite{Schafer}, see also \cite[p.149]{Albert65}. $(\mathbb{A},\bullet)$ denotes an algebra with a multiplication $\bullet$ on a vector space $\mathbb{A}$. All algebras below are assumed to be commutative but maybe nonassociative.

If $(\mathbb{B},\bullet)$ and $(\mathbb{C},\bullet)$ are ideals of an algebra $(\mathbb{A},\bullet)$ such that $\mathbb{A}$ as a vector space is the direct sum of $\mathbb{B}$ and $\mathbb{C}$ then $(\mathbb{A},\bullet)$ is called the \textit{direct sum}:
$$
(\mathbb{A},\bullet)=(\mathbb{B},\bullet)\oplus (\mathbb{C},\bullet).
$$
Note that in this case $\mathbb{B}\bullet \mathbb{C}=0$, so that $(\mathbb{B},\bullet)$ and $(\mathbb{C},\bullet)$ are \textit{orthogonal}.

Given any two arbitrary algebras  $(\mathbb{B},\star)$ and $(\mathbb{C},*)$ over a field $\Field$, one can construct an algebra $(\mathbb{A},\bullet)$ over $\Field$ such that $(\mathbb{A},\bullet)$ is the direct sum $(\mathbb{A},\bullet)=(\mathbb{B}',\bullet)\oplus (\mathbb{C}',\bullet)$ of ideals
$(\mathbb{B}',\bullet)$ and $(\mathbb{C}',\bullet)$ which are isomorphic respectively to $(\mathbb{B},\star)$ and $(\mathbb{C},*)$: $\mathbb{A}$ as a vector space  is the Cartesian product  of $\mathbb{B}$ and $\mathbb{C}$ with the multiplicative structure $\bullet$ defined by the coordinate-wise multiplication $(b_1,c_1)\bullet (b_2,c_2)=(b_1\star b_2, c_1* c_2)$ for elements $(b_1,c_1),(b_2,c_2)\in \mathbb{B}\times \mathbb{C}$. Then, for example, $(\mathbb{B},\star)$ is isomorphic to the ideal $(\mathbb{B}',\bullet)=((\mathbb{B},0),\bullet)$ of $(\mathbb{A},\bullet)$. The resulting algebra is called the direct product of algebras $(\mathbb{B},\star)$ and $(\mathbb{C},*)$. In the above notation,
$$
(\mathbb{B},\star)\times (\mathbb{C},*)\cong (\mathbb{B}\times \mathbb{C},\bullet)\cong (\mathbb{B}',\bullet)\oplus (\mathbb{C}',\bullet)
$$
As in the case of vector spaces, the notion of direct sum extends to an arbitrary set of summands. We shall have occasion to use only finite direct sums.

In particular, rhe field  $\Field$ is a commutative associative algebra over itself, denoted by $(\Field,\cdot)$. The direct summa of $n\ge1$ copies  of $\Field$ with the coordinate-wise multiplication is denoted by $(\Field^{n} ,\bullet)$. The $\bullet$-idempotents
\begin{equation}\label{standard}
e_i=(0,\ldots,1,\ldots,0), \qquad 1\le i\le n
\end{equation}
form the standard basis of  $\Field^{n}$. Let $x_i$ denote the corresponding coordinate of $x\in \Field^{n}$. Given any index $j\in \{1,2,\ldots,n\}$, the subspaces
\begin{equation}\label{Vj}
\Sspan{e_i}:=\{x\in \Field^{n}: \, x_i=0\, \text{ for all $i\ne j$}\}
\end{equation}
are (pairwise orthogonal) ideals of $(\Field^{n},\bullet)$ and
$$
(\Field^{n} ,\bullet)=\bigoplus_{j=1}^n \Sspan{e_j}.
$$

Below we shall need the following observation.
\begin{proposition}\label{pro:decompose}
Let $(\Alg,\bullet)=\bigoplus_{j=1}^r (\mathbb{B}_j,\bullet)$, where $(\mathbb{B}_j,\bullet)$ are ideals. Then $c$ is an idempotent in $(\Alg,\bullet)$ if and only if there are (uniquely determined) pairwise orthogonal idempotents $c_j\in (\mathbb{B}_j,\bullet)$ such that $c=\sum_{j=1}^rc_j$. Moreover, in this case
\begin{equation}\label{Lcgen}
\det (\lambda\mathds{1}_\Alg)-L_\bullet(c))=\prod_{j=1}^r\det (\lambda\mathds{1}_{\mathbb{B}_j}-L_\bullet(c_j))
\end{equation}
The algebra $(\Alg,\bullet)$ is generic if and only each ideal $(\mathbb{B}_j,\bullet)$ is so.
\end{proposition}

\begin{proof}
It follows for the assumptions that $\mathbb{B}_i\bullet \mathbb{B}_j=0$ for $i\ne j$. Let $x\in \Alg$ and let $x=\sum_{j=1}^rx_j$ be the corresponding orthogonal decomposition, $x_j\in \mathbb{B}_j$. Then $x\bullet x=\sum_{j=1}^r x_j\bullet x_j$, where $x_j\bullet x_j\in \mathbb{B}_j$ and $x_i\bullet x_j=0$  whenever $i\ne j$. It follows that $x$ is an idempotent on $(\Alg,\bullet)$ if and only if $x_j$ is an idempotent in $(\mathbb{B}_j,\bullet)$ for all $1\le j\le r$, thus implying the first part of the proposition.

Next, let $(\Alg,\bullet)$ be generic. Then $(\Alg,\bullet)$ contains exactly $2^n$ distinct regular idempotents, where $n=\dim \Alg$. It follows from \eqref{Lcgen} that all idempotents in each ideal $\mathbb{B}_j$ are regular. If $\dim \mathbb{B}_j=n_j$ then $n=n_1+\ldots +n_r$. It follows from the above idempotent decomposition  that $m_1\cdot \ldots \cdot m_r=2^n$, where $m_j$ is the cardinality of the set of idempotents in $\mathbb{B}_j$, hence $m_j=2^{k_j}$ for some nonnegative integer $k_j$. On the other hand by the Bezout inequality \eqref{Bezout} we have $m_j=2^{k_j}\le 2^{n_j}$, hence $2^n=m_1\cdot \ldots \cdot m_r\le 2^{n_1+\ldots +n_r}=2^n$ which implies that in fact $m_j=2^{k_j}= 2^{n_j}$, i.e. each $\mathbb{B}_j$ is generic. In the converse direction, if each ideal $\mathbb{B}_j$ is generic then $(\Alg,\bullet)=\bigoplus_{j=1}^r (\mathbb{B}_j,\bullet)$ has by Proposition~\ref{pro:decompose} at least $2^{n_1}\cdot \ldots 2^{n_r}=2^n$ distinct regular idempotents, hence it is generic.
\end{proof}

We also fix some standard terminology and facts from permutation theory. Any permutation $\sigma\in S_n$ can be written in \textit{cyclic form} (or a disjoint cycle decomposition)
\begin{equation}\label{cyclic1}
\sigma=\sigma_1\ldots \sigma_r \in S_n.
\end{equation}
For any $1\le i,j\le r$, the \textit{cycles} $\sigma_i$ and $\sigma_j$ commute, therefore the order in \eqref{cyclic1} in inessential. Then cycles $\sigma_j$ can be naturally thought of as orbits of a faithful action of the cyclic group $\langle \sigma\rangle$ generated by $\sigma$ on the et of indices $\bar n:=\{1,\ldots,n\}$. To differ a cycle $\sigma_j$ as a group element and as an orbit, we denote the latter by $[\sigma_j]\subset \bar n$. By $|\sigma_i|$ we denote the \textit{length} of the cycle $\sigma_i$, i.e. the cardinality of the orbit $[\sigma_i]$. The \textit{type of a permutation} is the integer partition of $n$,
\begin{equation}\label{alls}
|\sigma_1|+\ldots+|\sigma_r|=n,
\end{equation}
formed from the cycle; we write it by $\type(\sigma)=(|\sigma_1|,\ldots,|\sigma_r|)$. 

\begin{definition}\label{def:adm}
Given $\sigma\in S_n$ let  $\sigma=\sigma_1\ldots \sigma_r$ be its disjoint cycle decomposition and $s_i=|\sigma_i|$. A field $\Field$ will be said $\sigma$-\textbf{admissible}, if it is a splitting field for all polynomials $z^t-1$, where $t\in S:=\{s_1,\ldots,s_r, 2^{s_1}-1,\ldots, 2^{s_r}-1\}$. If $\Field$ has a finite characteristic, to avoid complications,  we additionally assume that the characteristic is co-prime with all numbers in $S$,
\end{definition}

\section{Inner isotopes}\label{sec:inner}

Two algebras $(\Alg,\diamond)$ and $(\mathbb{B}, \cdot)$ are called \textbf{isotopic}, if there is an \textit{isotopism} $(\alpha,\beta,\gamma)$, i.e. a triple of nondegenerate linear maps $\Alg\to \mathbb{B}$ such that $\alpha(x)\cdot \beta(y)=\gamma(x\diamond y)$ holds for any $x,y\in \Alg$. If $\alpha=\beta$ the isotopy is called \textit{strong}. If $\Alg=\mathbb{B}$ and $\gamma=\mathds{1}_\Alg$ in the identity map, then an isotopy is called a \textbf{principal autotopy}.

\begin{definition}
Given an  algebra $(\Alg,\bullet)$ and an algebra endomorphism  $h\in \End(\Alg,\bullet)$, we define a new algebra $(\Alg,\bullet_h)$ on the vector space $\Alg$ by
\begin{equation}\label{hdef}
x\bullet_h  y=h(x)\bullet h(y)=h(x\bullet y).
\end{equation}
The new algebra $(\Alg,\bullet_h)$ is called a \textbf{weak inner isotope} of $(\Alg,\bullet)$. If $h\in \Aut(\Alg,\bullet)$,  $(\Alg,\bullet_h)$ is called an \textbf{inner isotope}.
\end{definition}


The above  definition is a  particular case of a \textbf{strong principal autotopy}; more precisely,  the algebra $(\Alg,\bullet_h)$ is an principal isotopic of $(\Alg,\bullet)$ with $(h,h,\mathds{1}_\Alg)$ where  the map $h$ is an $\bullet$-\textit{algebra homomorphism}.

\begin{remark}
In the converse direction, if an endomorphism $h$ is invertible then  $(\Alg,\bullet)$ is  an inner isotope of $(\Alg,\bullet_h)$. Indeed, by virtue of \eqref{hdef}, $x\bullet y=h^{-1}(x\bullet_h y)$, therefore the linear endomorphism  $h^{-1}$ is also an $\bullet_h$-algebra endomorphism.
\end{remark}

We distinguish  the case when $h$ is an bijective, i.e. $h$ is an automorphism of $(\Alg, \bullet)$ and $(\Alg, \bullet_h)$ is an inner isotopy. In this case, if $h,f\in \Aut(\Alg,\bullet)$ then $(\mathds{1}_\Alg,\mathds{1}_\Alg, f\circ h^{-1})$ is an isotopy between the corresponding inner isotopes:
\begin{equation}\label{fh}
x\bullet_f  y=f(x\bullet y)=f\circ h^{-1}(x\bullet_h  y),
\end{equation}
where $\circ$ here and in what follows denote the composition of two maps.

Comparing \eqref{fh} with \eqref{hdef} a natural question arises: when the algebra  $(\Alg,\bullet_f)$ is an \textit{inner} isotope of $(\Alg,\bullet_h)$? Note that $f\circ h^{-1}\in \Aut(\Alg,\bullet)$, but it is not clear whether $f\circ h^{-1}\in \Aut(\Alg,\bullet_h)$, see the diagram below
$$
\begin{tikzcd}
 & (\Alg,\bullet) \arrow[swap, "h",dl] \arrow[dr,"f"] \\
(\Alg,\bullet_h)  \arrow[rr,"f\circ h^{-1}"] && (\Alg,\bullet_f)
\end{tikzcd}
$$
The next proposition reveals that this is true for commuting automorphisms only.

\begin{proposition}\label{pro:conj}
Let $(\Alg,\bullet)$ satisfy
\begin{equation}\label{simple}
\Alg^{\bullet2}:=\Alg\bullet \Alg=\Alg
\end{equation}
and let $h,f\in \Aut(\Alg,\bullet)$. Then $(\Alg,\bullet_f)$ is an inner isotope of $(\Alg,\bullet_h)$ if and only if $f$ and $h$ commute.
\end{proposition}

\begin{proof}
First  suppose  that $(\Alg,\bullet_f)$ is an inner isotope of $(\Alg,\bullet_h)$, i.e.
\begin{equation}\label{findent}
x\bullet_f  y=g(x\bullet_h  y)=g(x)\bullet_h  g(y), \qquad \forall x,y\in \Alg,
\end{equation}
holds for some $\bullet_h$-algebra endomorphism $g$. Comparing this with \eqref{fh}, we obtain
$$
f(x\bullet y)=f\circ h^{-1}(x\bullet_h  y)=g(x\bullet_h  y)=(g\circ h)(x \bullet y).
$$
Since the latter holds for any $x,y\in \Alg$ we conclude that $f=g\circ h$ on $\Alg^{\bullet2}$, and thus by \eqref{simple} on $\Alg$. Therefore $g=f\circ h^{-1}$. In particular, $g$ is an $\bullet$-automorphism too. Therefore, using the second identity in \eqref{findent} we get
\begin{align*}
f(x\bullet  y)&=g(x)\bullet_h  g(y)=h(g(x)\bullet g(y))\\
&=(h\circ g(x))\bullet (h\circ g(y))\\
&=(h\circ g)(x \bullet y)\qquad \forall x,y\in \Alg,
\end{align*}
implying that $f=h\circ g=h\circ f\circ h^{-1}$ on $\Alg^{\bullet2}=\Alg$, hence $f\circ h=h\circ f$ as desired.

Conversely, if $f\circ h=h\circ f$ then using \eqref{fh}
\begin{equation}\label{fh1}
x\bullet_f  y=g(x\bullet_h  y),
\end{equation}
where $g=f\circ h^{-1}= h^{-1}\circ f\in \Aut(\Alg,\bullet)$. We have
$$
g(x\bullet_h  y)=g\circ h(x\bullet  y)=f\circ h^{-1}\circ h(x\bullet  y)=f(x\bullet  y),
$$
and on the other hand,
$$
g(x)\bullet_h  g(y)=h\circ (g(x)\bullet  g(y))=(h\circ g)(x\bullet  y)=f(x\bullet  y),
$$
implying $g(x\bullet_h  y)=g(x)\bullet_h  g(y)$, hence $g\in \Aut(\Alg,\bullet_h)$, i.e. it follows from \eqref{fh1} that $(\Alg,\bullet_f)$ is an inner isotope of $(\Alg,\bullet_h)$.
\end{proof}

The next two propositions explain when two inner isotopes are isomorphic.

\begin{proposition}
\label{pro:isomA}
Let $f:(\Alg,\bullet)\to (\Alg',\bullet')$ be an algebra isomorphism and let $h\in \Aut(\Alg,\bullet)$. Then $h':=f\circ h\circ f^{-1}\in \Aut(\Alg',\bullet')$ and $f:(\Alg,\bullet_h)\to (\Alg',\bullet'_{h'})$ is an algebra isomorphism.
\end{proposition}

\begin{proof}
We have $f(x\bullet y)=f(x)\bullet' f(y)$ and $h(x)\bullet h(y)=h(x\bullet y)$ for any $x,y\in \Alg$, therefore
\begin{align*}
f\circ h\circ f^{-1}( x'\bullet' y')&=f\circ h(f^{-1}( x')\bullet f^{-1}(y'))\\
&=f(h(f^{-1}( x'))\bullet h(f^{-1}(y')))\\
&=f(h(f^{-1}( x'))\bullet' f(h(f^{-1}(y')),
\end{align*}
readily implying that  $h':=f\circ h\circ f^{-1}\in \Aut(\Alg',\bullet')$. Furthermore,
$$
f(x\bullet_h  y)=f\circ h(x\bullet  y)=f\circ h\circ f^{-1}(f(x)\bullet'  f(y))=f(x)\bullet'_{h'}  f(y)
$$
implying that $f:(\Alg,\bullet_h)\to (\Alg',\bullet'_{h'})$ is an algebra isomorphism.

\end{proof}

\begin{proposition}
\label{pro:isomiso}
Let $h,f\in \Aut(\Alg,\bullet)$. If $h$ and $f$ conjugate in $\Aut(\Alg,\bullet)$ then $(\Alg,\bullet_h)$ is isomorphic to $(\Alg,\bullet_f)$.
In the converse direction, if $g:(\Alg,\bullet_h)\to (\Alg,\bullet_f)$ is an isomorphism and $g\in \Aut(\Alg,\bullet)$  then $h$ and $f$ conjugate in $\Aut(\Alg^{\bullet2},\bullet)$.
\end{proposition}

\begin{proof}
If $f$ and $h$ conjugate in $\Aut(\Alg,\bullet)$ then $f=g\circ h\circ g^{-1}$ for some $g\in\Aut(\Alg,\bullet)$, hence
$$
g(x\bullet_h y)=g\circ h(x\bullet y)=f\circ g(x\bullet y)=f(g(x)\bullet g(y))=g(x)\bullet_f g(y),
$$
hence $g:(\Alg,\bullet_h)\to (\Alg,\bullet_f)$ is an algebra isomorphism. Conversely, by our assumptions we have
$$
g\circ h(x\bullet y)=g(x\bullet_h y)=g(x)\bullet_f g(y)=f(g(x)\bullet g(y))=f\circ g(x\bullet y),
$$
hence $g\circ h=f\circ g$ on $\Alg^{\bullet2}$.
\end{proof}

\medskip
\begin{corollary}\label{cor:conj}
Let $(\Alg,\bullet)$ satisfy \eqref{simple} and $h,f\in \Aut(\Alg,\bullet)$. Then $(\Alg,\bullet_h)$ is isomorphic to $(\Alg,\bullet_f)$ if and only if $h$ and $f$ conjugate in $\Aut(\Alg,\bullet)$.
\end{corollary}

\section{Inner isotopes of commutative associative algebras}\label{sec:comm}
So far, we have not specified any additional algebraic structure on $(\Alg,\bullet)$. Below we shall focus on the simplest case when the original algebra $(\Alg, \bullet)$ is commutative and associative. Sometimes we shall also require that the algebra $(\Alg, \bullet)$ is unital. Note that any unital algebra satisfies automatically \eqref{simple}.


If $(\Alg, \bullet)$ is a commutative  associative algebra then any inner isotope $(\Alg, \bullet_h)$ is obviously commutative but it maybe non-associative, because
$$
x\bullet_h  (y\bullet_h  z)=h(x\bullet h(y\bullet z))=h(x)\bullet h^2(y)\bullet h^2(z)
$$
and
$$
(x\bullet_h  y)\bullet_h  z=h^2(x)\bullet h^2(y)\bullet h(z)
$$
are not equal in general. On the other hand, such an inner isotope is nearly associative, namely, it is medial. We recall the definition.

\begin{definition}
An algebra $(\Alg,\bullet)$ is called \textit{medial} if
\begin{equation}\label{medial}
(x\bullet   y)\bullet   (z\bullet   w)=
(x\bullet   z)\bullet   (y\bullet   w), \qquad \forall x,y,z,w\in \Alg.
\end{equation}
\end{definition}

An important corollary of the definition is
$$
(x\bullet   y)\bullet   (x\bullet   y)=
(x\bullet   x)\bullet   (y\bullet   y).
$$
which immediately implies

\begin{proposition}\label{pro:medidm}
Product of two idempotents in a medial algebra is an idempotent again.
\end{proposition}

\begin{proof}
If $c_i\bullet c_i=c_i$, $i=1,2$ then
$$
(c_1\bullet c_2)\bullet(c_1\bullet c_2)=(c_1\bullet c_1)\bullet (c_2\bullet c_2)=c_1\bullet c_2.
$$
\end{proof}

\begin{proposition}\label{pro:medial}
If an algebra $(\Alg,\bullet)$ is commutative and associative, $h\in \End(\Alg,\bullet)$, then the inner isotope $(\Alg, \bullet_h)$ is a commutative medial algebra.
Furthermore, if $h\in \Aut(\Alg,\bullet)$ and $(\Alg,\bullet)$ is additionally a unital algebra with unity $e$ then $e$ is an idempotent in $(\Alg, \bullet_h)$ and $L_{\bullet_h}(e)$ is an invertible operator.

\end{proposition}

\begin{proof}
We have
\begin{align*}
(x\bullet_h  y)\bullet_h  (z\bullet_h  w)&=h((x\bullet_h  y)\bullet_h  (z\bullet_h  w))\\
&=h(h(x\bullet y)\bullet h(z\bullet w))\\
&=h^2(x\bullet y\bullet z\bullet w),
\end{align*}
where the right hand side is symmetric under any permutation of factors, implying \eqref{medial}.

Next, assume that $(\Alg,\bullet)$ is additionally a unital algebra with unity $e$. Since an automorphism stabilizes the unity element, we have
$$
e\bullet_h e=h(e\bullet e)=h(e)=e
$$
therefore $e$ becomes an idempotent in $(\Alg, \bullet_h)$. Furthermore, $e\bullet_h x=h(e\bullet x)=h(x)=0$ if and only if $x=0$, thus $L_{\bullet_h}(e)$ is an invertible operator.

\end{proof}


\begin{proposition}\label{pro:unital}
Any associative commutative algebra  is medial. A unital commutative medial algebra  is  associative.
\end{proposition}
\begin{proof}
Indeed, if $(\Alg,\bullet)$ is an associative commutative algebra then it is medial:
$$(x\bullet y)\bullet (z\bullet w)=x\bullet y\bullet z\bullet w=x\bullet z\bullet y\bullet w=(x\bullet z)\bullet (y\bullet w).
$$
On the other hand, if $(\Alg,\bullet)$ is a unital commutative medial algebra and $e$ is the algebra unity then
$$
x\bullet (y\bullet z)=(e\bullet x)\bullet (y\bullet z)=(e\bullet z)\bullet (x\bullet y)=z\bullet (x\bullet y)=(x\bullet y)\bullet z,
$$
hence $(\Alg,\bullet)$ is associative.
\end{proof}

We refer to \cite{KrTk23a} for general medial algebras and their spectral theory.

\section{Categories of calibrated medial algebras}\label{sec:categor}


We denote by $\Idm^\times(\Alg,\ast)$ the subset of nonzero idempotents in an algebra $(\Alg,\ast)$ such that the left multiplication operator $L_{\ast}(c)$ is invertible.

\begin{definition}
A  medial algebra $(\Alg,\ast)$ is called \textbf{special} if the set $\Idm^\times(\Alg,\ast)$ is non-empty.
\end{definition}

The terminology `special' here correspond exactly to what is called `medial class (iii) algebras' in \cite{KrTk23a}.

It follows from Proposition~\ref{pro:medial} above that if $(\Alg,\bullet)$ unital commutative associative  algebra then its unity element $e$ of  becomes an idempotent in the inner isotope  $(\Alg,\bullet_h)$, $h\in \Aut(\Alg,\bullet)$ and furthermore  $e\in \Idm^\times(\Alg,\bullet_h)$. This proves

\begin{corollary}\label{cor:med}
Any inner isotope $(\Alg,\bullet_h)$ of a unital commutative associative algebra $(\Alg,\bullet)$ is a special medial algebra.
\end{corollary}

Idempotents in $\Idm^\times(\Alg,\ast)$ are distinguished in many aspects and a particular choice of such an idempotent can be thought of as a \textit{calibration} or \textit{pointing} of an algebra. To put this observation into an appropriate context, we define a concept  of calibrations for special medial algebras and unital commutative associative algebras.
An important ingredient in our constructions is  Kaplansky's trick \cite{Kaplansky}, \cite{Peter2000}.


\begin{definition}
A special (commutative) medial algebra $(\Alg, \ast)$ with a distinguished idempotent $c\in \Idm^\times(\Alg,\ast)$ is called \textit{calibrated} and denoted by $(\Alg, \ast,c)$. An algebra homomorphism  between two special medial algebras $f:(\Alg, \ast,a)\to  (\mathbb{B}, \circledast,b)$  is called a \textit{calibrated} if $f(a)=b$.
\end{definition}


\begin{definition}
A unital commutative associative algebra $(\Alg, \diamond, e)$ with algebra unity $e$ and a distinguished automorphism $h\in \Aut(\Alg, \diamond, e)$ is called \textit{calibrated} and denoted by $(\Alg, \diamond, e,h)$. An algebra homomorphism  between two calibrated commutative associative algebras $f:(\Alg, \diamond, e,h)\to  (\mathbb{A}', \diamondsuit, e',h')$ is called a \textit{calibrated homomorphism} if $h'\circ f=f\circ h$ (note that $f(e)=e'$).
\end{definition}

We denote two calibrated isomorphic algebras  by $\Alg\congp \Alg'$. Of course, $\Alg\congp \Alg'$ implies that $\Alg\cong \Alg'$.
\smallskip

Denote by $\frak{M}$ (respectively by $\mathfrak{A}$) the class of calibrated special commutative medial algebras (respectively  calibrated commutative associative unital algebras). These classes  are categories in an obvious way, where the corresponding morphisms are calibrated homomorphisms.

\begin{theorem}\label{the:categor}
The functor $\Phi:\frak{A}\to \frak{M}$ given by $\Phi(\Alg, \diamond, e, h)=(\Alg, \ast, e)$, where $x\ast y=h(x\diamond y)$ and $e\in \Idm^\times(\Alg,\ast,e)$,
is a category isomorphism. The inverse functor $\Psi=\Phi^{-1}$ is given by $\Psi(\Alg, \ast, c)=(\Alg, \diamond, c,h)$, where $x\diamond y=L_\ast^{-1}(x\ast y)$, $c$ is a unity in $(\Alg, \diamond)$ and $h=L_\ast(c)\in \Aut(\Alg, \diamond)$.
\end{theorem}




\begin{proof}
The proof is divided into three steps.

\textbf{Step 1.} The map $\Phi$ is a functor $\mathfrak{A}\to \mathfrak{M}$.

Our first claim is that given $(\Alg, \diamond,e,h)\in \mathfrak{A}$, a new multiplication on $\Alg$ by
\begin{equation}\label{circ0}
x\ast y=h(x\diamond y)=h(x)\diamond h(y)
\end{equation}
defines  a calibrated medial algebra $\Alg,\ast,e)\in \mathfrak{M}$ of class (iii) with an idempotent $e\in \Idm^\times(\Alg,\ast)$. Indeed, since $h\in\Aut(\Alg, \diamond)$ we have
$$
(x\ast y)\ast (z\ast w)=h((x\ast y)\diamond (z\ast w))=
h(h(x\diamond y)\diamond h(z\diamond w))=(h\circ h)(x\diamond y\diamond z\diamond w)
$$
where the right hand side is obviously symmetric for any permutation of factors. Since $e$ is the unitity element and $h$ is an algebra automorphism in $(\Alg,\diamond)$  then $h(e)=e$, hence $e\ast e=h(e\diamond e)=e$, i.e. $e$ is an idempotent in $(\Alg,\ast)$. Also, $e\ast x=h(e\diamond x)=h(x)$, hence $L_{\ast}(e)=h$ is a bijection, i.e. $e\in \Idm^\times(\Alg,h,\ast)$.

Next, we prove that $\Phi$ acts correspondingly on morphisms. To this end consider a calibrated $\frak{A}$-algebra homomorphism  $f:(\Alg, \diamond, e,h)\to  (\mathbb{A}', \diamondsuit, e',h')$. Then $f$ is an algebra homomorphism and $h'\circ f=f\circ h$. Let
$\Phi(\Alg, \diamond, e,h)=(\Alg, \ast, e )$ and $\Phi(\Alg', \diamondsuit, e',h')=(\Alg', \circledast, e' )$. We consider $\Phi(f)=f$. Then
$f(e)=e'$ and
$$
f(x\ast y) =(f\circ h)(x\diamond y)=(h'\circ f)(x \diamond y)=h'(f(x) \diamondsuit f(y))=f(x) \circledast f(y),
$$
hence $f:(\Alg, \ast, e )\to (\Alg', \circledast, e' )$ is a calibrated $\mathfrak{M}$-algebra homomorphism. The fact that $\Phi$ preserves identity morphisms and composition of morphisms trivially follows by its definition.

\textbf{Step 2.} The map $\Psi$ is a functor $\mathfrak{M}\to \mathfrak{A}$.

Let $(\Alg, \ast,c)\in \mathfrak{M}$. Define a new multiplication on $\Alg$ by
\begin{equation}\label{circ}
x\diamond y=L_\ast(c)^{-1}(x\ast y).
\end{equation}
The algebra  $(\Alg,\diamond)$ is commutative and $c\diamond x=L_\ast(c)^{-1}(c\ast x)=L_\ast(c)^{-1}L_\ast(c)x=x$, hence $c$ is a unit in $(\Alg,\diamond)$. By \ref{ideal1} in Proposition~\ref{pro:idem},  $L_\ast(c)^{-1}$ is an algebra isomorphism of $(\Alg, \ast,c)$ hence
$$
(x\diamond y)\diamond (z\diamond w)=L_\ast(c)^{-1}\bigl(L_\ast(c)^{-1}(x\ast y)\ast L_\ast(c)^{-1}(z\ast w)\bigr)=L_\ast(c)^{-2}((x\ast y)\ast(z\ast w))
$$
which implies that $(\Alg,\diamond)$ is medial. By Proposition~\ref{pro:unital}, $(\Alg,\diamond)$ is  associative. Since $L_\ast(c)$ is an algebra isomorphism of $(\Alg, \ast,c)$, it is a bijection. Furthermore,
\begin{align*}
L_\ast(c)(x\diamond y)&=L_\ast(c)L_\ast(c)^{-1}(x\ast y)=x\ast y\\
L_\ast(c)(x)\diamond L_\ast(c)(y)&=L_\ast(c)^{-1}\bigl(L_\ast(c)(x)\ast L_\ast(c)(y)\bigr)=x\ast y,
\end{align*}
hence $L_\ast(c)(x\diamond y)=L_\ast(c)(x)\diamond L_\ast(c)(y)$, i.e. $L_\ast(c)\in \Aut(\Alg,\diamond,c)$. Then  $\Psi(\Alg, \ast,c) =(\Alg,\diamond,c,L_\ast(c))\in \mathfrak{A}$.

Next, we prove that $\Phi$ acts correspondingly on morphisms. Let $f:(\Alg, \ast, c)\to  (\mathbb{A}', \circledast, c')$ be  a calibrated $\frak{M}$-algebra homomorphism, i.e. $f(c)=c'$. Let
$\Psi(\Alg, \ast, c)=(\Alg, \diamond, c,L_\ast(c))$ and $\Psi((\mathbb{A}', \circledast, c')=(\Alg', \diamondsuit, c',L_\circledast(c'))$. Then $f:(\Alg, \diamond)\to (\Alg', \diamondsuit)$ is  vector space homomorphism and
$$
(L_\circledast(c')\circ f)(x)=c'\circledast f(x)=f(c)\circledast f(x)=f(c\ast x)=(f\circ L_\ast(c))(x),
$$
hence $L_\circledast(c')\circ f=f\circ L_\ast(c)$ and therefore $ f\circ L_\ast(c)^{-1}=L_\circledast(c')^{-1}\circ f$, and it follows that
\begin{align*}
f(x\diamond y)&=\bigl(f\circ L_\ast(c)^{-1}\bigr)(x\ast y)=\bigl(L_\circledast(c')^{-1}\circ f\bigr)(x\ast y)\\
&=L_\circledast(c')^{-1}(f(x)\circledast f(y))=f(x)\diamondsuit f(y)
\end{align*}
i.e. $f$ is a calibrated $\frak{A}$-algebra homomorphism. The fact that $\Psi$ preserves identity morphisms and composition of morphisms readily follows by its definition.

\textbf{Step 3.} $ \Psi\circ \Phi=\mathrm{id}_{\mathfrak{A}}.$

We have $\Phi(\Alg, \diamond, e, h)=(\Alg, \ast, e)$ and $\Psi(\Alg, \ast, e)=(\Alg,\diamondsuit,e,L_\ast(e))$, where
\begin{align*}
x\ast y&=h(x\diamond y), \\
x\diamondsuit y&=L_\ast(e)^{-1}(x\ast y).
\end{align*}
We have $e\ast z=h(e\diamond y)=h(e)\diamond h(z)=h(z)$, i.e. $L_\ast(e)=h$, therefore $x\diamondsuit y=L_\ast(e)^{-1}(x\ast y)=h^{-1}(x\ast y)=x\diamond y$. This implies that $\Psi(\Alg, \ast, e)=(\Alg,\diamond,e,h)$, hence $ \Psi\circ \Phi=\mathrm{id}_{\mathfrak{A}}$.
The theorem follows.
\end{proof}

We have several important corollaries of the above result.

\begin{corollary}\label{cor:psiphi}
$\Psi(\Alg)\congp\Psi(\Alg')$ if and only if $\Alg\congp \Alg'$ for $\Alg,\Alg'\in \mathfrak{M}$ and
$\Phi(\Alg)\congp\Phi(\Alg')$ if and only if $\Alg\congp \Alg'$ for $\Alg,\Alg'\in \mathfrak{A}$.
\end{corollary}

The following propositions describe how the functors $\Phi:\mathfrak{A}\to \mathfrak{M}$ and $\Psi:\mathfrak{M}\to \mathfrak{A}$ depend on a particular choice of calibrating.


\begin{proposition}\label{pro:M}
If $(\Alg,\ast)$ is a special commutative medial algebra and $c_1,c_2\in \Idm^\times(\Alg,\ast)$ then  $(\Alg, \ast, c_1)\congp (\Alg, \ast, c_2)$. In particular, given a special commutative medial algebra, there exists a unique calibrated isomorphy class of $\Alg$.
\end{proposition}

\begin{proof}
By \ref{ideal1} in Proposition~\ref{pro:idem}, $f=L_\ast(c_1)^{-1}L_\ast(c_2)$ is a $\ast$-algebra automorphism of $(\Alg,\ast)$. Furthermore,
$
f(c_1)=L_\ast(c_1)^{-1}L_\ast(c_2)(c_1)=L_\ast(c_1)^{-1}(c_1\ast c_2)=c_2,
$
hence $f$ a calibrated isomorphism of $f:(\Alg, \ast, c_1)\to (\Alg, \ast, c_2)$.
\end{proof}

\begin{proposition}\label{pro:A}
Let $(\Alg,\diamond)$ be a unital commutative associative algebra, $h_1,h_2\in \Aut(\Alg,\diamond)$. Then $(\Alg,\diamond,e,h_1)\congp (\Alg,\diamond,e,h_2)$ if and only if  $h_1$ and $h_2$ are conjugate in $\Aut(\Alg,\diamond)$.
\end{proposition}

\begin{proof}
By the definition, $(\Alg,\diamond,e,h_1)\congp (\Alg,\diamond,e,h_2)$ if and only if there exists an isomorphism $f:(\Alg,\diamond)\to (\Alg,\diamond)$ such that $h_2\circ f=f\circ h_1$, i.e. $f\in \Aut(\Alg,\diamond)$  and $h_2=f\circ h_1\circ f^{-1}$, which is equivalent to that  $h_1$ and $h_2$ are conjugate in $\Aut(\Alg,\diamond)$.
\end{proof}

\begin{corollary}\label{cor:poin}
Given a unital commutative associative  algebra $(\Alg,\diamond)$, there is a natural bijection between its calibrated isomorphy classes  and the conjugacy classes of its automorphism group:
$$
\Alg/_{\congp} = \Aut(\Alg)/_{\mathrm{conj}}
$$
\end{corollary}

\section{Idempotents in inner isotopes}\label{sec:idem}
We start with a general result which holds for any commutative medial algebra.
\begin{proposition}[\cite{KrTk23a}]
\label{pro:idem}
Let $(\Alg,\ast)$ be commutative medial algebra.
If $c_1,c_2\in \Idm(\Alg, \ast)$  then so is $c_1\ast c_2$. In other words, the set of all idempotents $\Idm(\Alg, \ast)\cup \{0\}$ is a multiplicative magma.
Furthermore, for any idempotent $c\in \Idm(\Alg, \ast)$:
\begin{enumerate}
\item\label{ideal1}
$L_{\ast}(c)$ is a $(\Alg, \ast)$-algebra endomorphism;
\item\label{ideal2} The $0$-Peirce subspace $\ker L_{\ast}(c)$ is an ideal of $(\Alg, \ast)$ and the image $L_{\ast}(c)(\Alg)$ is a subalgebra of $\Alg$;
\item\label{ideal4} The $1$-Peirce subspace $\{x\in \Alg:L_{\ast}(c)x=x\}$ is a  subalgebra of the image $L_{\ast}(c)(\Alg)$ and $\dim \Alg_c(1)\ge 1$;
\item\label{ideal5} For any idempotents $c_1, c_2\in \Idm(\Alg)$ the following composition rule holds:
\begin{equation}\label{c1c2}
L_{\ast}(c_2)L_{\ast}(c_1)=L_{\ast}(c_1\ast c_2)L_{\ast}(c_2)
\end{equation}
\end{enumerate}
\end{proposition}

\begin{proof}
The first claim is an immediate corollary of the medial magma identity \eqref{medial}. Furthermore, the multiplication operator $L_{\ast}(c):\Alg\to \Alg$ is linear and it follows from  \eqref{medial} that for any idempotent $c\in \Idm(\Alg,\ast)$
\begin{align*}
L_{\ast}(c)(x\ast y)&=c\ast(x\ast y)=(c\ast c)\ast(x\ast y)=(c\ast x)\ast(c\ast y)\\
&=L_{\ast}(c)x\,\ast L_{\ast}(c)y,
\end{align*}
hence $L_{\ast}(c)$ is an algebra endomorphism. As the kernel of a homomorphism, $\Alg_c(0)=\ker L_{\ast}(c)$ is an ideal and as the image of a homomorphism, $L_{\ast}(c)(\Alg)$ is a subalgebra. Further, $\Alg_c(1)=\{x:L_{\ast}(c)x=x\}$ is the set of fixed points of the algebra homomorphism $L_{\ast}(c)$, thus it is a subalgebra of $\Alg$. Since $L_{\ast}(c)$ stabilizes $\Alg_c(1)$, the latter is also a subalgebra of $L_{\ast}(c)(\Alg)$. Also, the one-dimensional subspace $
\Span{c}\subset L_{\ast}(c)(\Alg),$ hence $\dim \Alg_c(1)\ge 1$. Finally, \eqref{c1c2} follows from
\begin{align*}
L_{\ast}(c_2)L_{\ast}(c_1)x&=c_2\ast (c_1\ast x)
=(c_2\ast c_2)(c_1\ast x)
=(c_2\ast c_1)(c_2\ast x)\\
&=L_{\ast}(c_1\ast c_2)L_{\ast}(c_2)x.
\end{align*}
\end{proof}

Now let $(\Alg, \bullet)$ be a commutative algebra and $(\Alg, \bullet_h )$ its inner isotope, $h\in \Aut(\Alg, \bullet)$. Then $c$ is an idempotent in $(\Alg, \bullet_h )$ if $c=c\bullet_h  c$, which implies
\begin{equation}\label{iter1}
\Idm(\Alg,\bullet_h )\cup\{0\}=\{c\in \Alg: h(c\bullet c)=c\}.
\end{equation}

Combining Proposition~\ref{pro:medial} and Proposition~\ref{pro:idem} we obtain

\begin{corollary}\label{cor:mult}
Let $(\Alg,\bullet)$ be commutative associative algebra.
The set of all idempotents  $\Idm(\Alg, \bullet_h)\cup\{0\}$ is a multiplicative magma.
\end{corollary}

The latter makes it natural to ask when the set of all \textit{nonzero} idempotents $\Idm(\Alg, \bullet_h)$ is a quasigroup. Note that this property does not hold in general because the product of two idempotents may be zero. But, under some natural assumptions, one has the desired property.

\begin{corollary}\label{cor:divis}
If $(\Alg, \bullet)$ is a commutative  associative division algebra and $h\in \Aut(\Alg, \bullet)$ then $(\Alg, \bullet_h)$ is also a division algebra and  the set of nonzero idempotents $\Idm(\Alg, \bullet_h)$ is a commutative idempotent medial quasigroup.
\end{corollary}

\begin{proof}
Indeed,  given $x,y\in (\Alg, \bullet_h)$,  $x\bullet_h y=0$ if and only if $h(x\bullet y)=0$, where the latter by the bijectivity of $h$ is equivalent to $x\bullet y=0$, therefore $(\Alg, \bullet_h)$ is also a division algebra. This implies by Proposition~\ref{pro:idem} that for any $c_1,c_2\in \Idm(\Alg, \bullet_h)$ in fact $c_1\bullet_h c_2\in \Idm(\Alg, \bullet_h)$. Suppose that $c_1\bullet_h c_2=c_1\bullet_h c_3$ for some $c_1,c_2,c_3\in \Idm(\Alg, \bullet_h)$. Then $c_1\bullet_h (c_2- c_3)=0$, therefore $c_2- c_3=0$, hence $\Idm(\Alg, \bullet_h)$ is in fact a quasigroup, which is obviously commutative idempotent and medial, the proposition follows.
\end{proof}

\begin{proposition}\label{pro:quasi}
Let $(\Alg, \bullet)$ be a commutative  associative division algebra and $h\in \Aut(\Alg, \bullet)$. Then all idempotents $c\in \Idm(\Alg, \bullet_h)$ have the same characteristic polynomial.
\end{proposition}

\begin{proof}
By Corollary~\ref{cor:divis}, if $c_1,c_2\in \Idm(\Alg, \bullet_h)$ then $c_1\bullet_h c_2\in \Idm(\Alg, \bullet_h)$ and $L_{\bullet_h}(c_1)$ is invertible, hence   by  Proposition~\ref{pro:idem} we obtain
$$
L_{\bullet_h}(c_1\bullet_h c_2)=L_{\bullet_h}(c_1)L_{\bullet_h}(c_2)L_{\bullet_h}(c_1)^{-1}.
$$
and similarly
$$
L_{\bullet_h}(c_2\bullet_h c_1)=L_{\bullet_h}(c_2)L_{\bullet_h}(c_1)L_{\bullet_h}(c_2)^{-1}.
$$
Since $c_1\bullet_h c_2=c_2\bullet_h c_1$ we obtain from the last two relations that the characteristic polynomials of $L_{\bullet_h}(c_2)$ and $L_{\bullet_h}(c_1)$ are equal.
The proposition follows.
\end{proof}

\begin{proposition}
Let $(\Alg, \bullet)$ be a  unital commutative  associative division algebra with unity $e$ and let $h\in \Aut(\Alg, \bullet)$ be an automorphism of finite order $d$. Then $c^{\bullet (2^d-1)}=e$ for any $c\in \Idm(\Alg,\bullet_h )$.
\end{proposition}

\begin{proof}
By \eqref{iter1} $c=c \bullet_h c=h(c\bullet c)=h(c^{\bullet2})$, therefore for all $k=1,2,\ldots$
\begin{align*}
c&=h(c)\bullet h(c)=h(h(c^{\bullet2}))\bullet h(h(c^{\bullet2}))=h^2(c^{\bullet 2^2})=\ldots=h^k(c^{\bullet 2^k}).
\end{align*}
Since $h^d=\mathrm{id}$, we obtain $c =h^{d}(c^{\bullet 2^d})=c^{\bullet 2^d}$,
i.e. $c\bullet (c^{\bullet (2^d-1)}-e)=0$, hence by the assumptions $(\Alg,\bullet)$ does not contain  divisors of zero, we conclude that $c^{\bullet (2^d-1)}=e$.
\end{proof}



\section{Inner isotopes of a quotient polynomial algebra}\label{sec:auto}\label{sec:tau}

It worthy to point out that although the algebras $\Field[z]/P$ considered in the introduction  originate from polynomials, their multiplicative structure does not depend on a particular choice of a polynomial or set of its roots, but only on the dimension $n$ (i.e. the number of \textit{distinct} roots)  and the choice of a conjugacy class of permutation $\sigma\in S_n$. Furthermore, under the splitting condition, the algebra $\Field[z]/P$ turns out isomorphic the direct product algebra of $n=\deg P$ copies of the ground field $\Field$, the theorem below claims.

\begin{theorem}\label{thA}
Let a polynomial $P\in \Field[z]$ split over $\Field$ and have $n=\deg P$ distinct roots. Then
\begin{equation}\label{isom}
\Field[z]/P\cong
(\Field^{n} ,\bullet).
\end{equation}
The automorphism group
$$
\Aut(\Field[z]/P)\cong \Aut(\Field^{n} ,\bullet)\cong S_n
$$
is isomorphic to the symmetric group $S_n$.
Furthermore, two inner isotopes of $(\Field^{n} ,\bullet_\sigma)$ and $(\Field^{n} ,\bullet_\tau)$, $\sigma, \tau\in S_n$ are isomorphic if and only if $\sigma$ and $\tau$ conjugate in $S_n$.
\end{theorem}

\begin{proof}[Proof of Theorem~\ref{thA}]
Under our assumptions, $P=(z-a_1)\ldots(z-a_n)$, where $\{a_1,\ldots, a_n\}$ are the distinct roots of $P$. Then the Chinese Remainder Theorem \cite[Propoisition~15]{Dummit} gives an explicit isomorphism
\begin{align*}
\Field[z]/P&\cong
\Field[z]/(z-a_1)\cdot \ldots\cdot (z-a_n)\\
&\cong (\Field[z]/(z-a_1))\times \ldots\times (\Field[z]/(z-a_n))\\
&\cong \underbrace{(\Field,\cdot)\times\ldots \times (\Field,\cdot)}_{n\text{ times}}\\
&\cong (\Field^{n} ,\bullet).
\end{align*}
implying \eqref{isom}, where $\bullet$ is the standard coordinate-wise multiplication on $\Field^{n} $. Let $e_i=(0,\ldots,1,\ldots,0)$, $1\le i\le n$ be the standard basis of $\Field^{n} $ (see section~\ref{sec:prelim}). Then it follows that $e_i\bullet e_i=e_i$, hence $e_i$ are idempotents of $(\Field^{n} , \bullet)$, and, moreover, the partial sums
\begin{equation}\label{primitive1}
e_{I}:=\sum_{i\in I} e_i,\qquad I\in  2^{\{1,2,\ldots , n\}}
\end{equation}
are the only nonzero idempotents in $\Idm(\Field^{n} ,\bullet)$. Furthermore, since $e_i\bullet e_j=0$ for any $1\le i< j\le n$, $\{e_i\}_{1\le i\le n}$ are the only  primitive (i.e. indecomposable) idempotents  in $\Idm(\Field^{n} ,\bullet)$. If $\phi\in \Aut(\Field^{n} ,\bullet)$ is an algebra automorphism  then $\phi(x\bullet x)=\phi(x)\bullet \phi(x)$, hence $\phi$ is a permutation of the set of nonzero idempotents $\Idm(\Field^{n} ,\bullet)$. Since $\phi(x+ y)=\phi(x)+ \phi(y)$, $\phi$ preserve primitive idempotents, thus $\phi$ is  a permutation of the set  $\{e_i\}_{1\le i\le n}\subset \Idm(\Field^{n} ,\bullet)$. This implies that $\Aut(\Field^{n} ,\bullet)$ is a subgroup of $S_n$. On the other hand, if $\sigma \in S_n$ is an arbitrary permutation, then the linear map
\begin{equation}\label{psiiso}
\psi_\sigma: (x_1,\ldots,x_n)\to (x_{\sigma(1)},\ldots, x_{\sigma(n)})
\end{equation}
is  an isomorphism of $(\Field^{n} ,\bullet)$ implying that in fact $\Aut(\Field^{n} ,\bullet)\cong S_n$. Finally, since $(\Field^{n} ,\bullet)$ trivially satisfies \eqref{simple}, we conclude by Corollary~\ref{cor:conj} that  two inner isotopes of $(\Field^{n} ,\bullet_\sigma)$ and $(\Field^{n} ,\bullet_\tau)$, $\sigma, \tau\in S_n$ are isomorphic if and only if $\sigma$ and $\tau$ conjugate in $S_n$.
\end{proof}

\begin{remark}\label{rem:cyc}
It follows from Theorem~\ref{thA} that  it suffices to consider some specific polynomial $P(z)$ in each degree $n$. The circular polynomials $P(z)=z^n-1$ are distinguished in many respects. The corresponding polynomial quotient algebra $(\Field^n,\bullet):=(\Field[z]/(z^n-1),\bullet)$ and its inner isotopes $(\Field^n,\bullet_\sigma)$ have originally been introduced and studied in \cite{KrTk23a} in the particular case when $\sigma=(2\,3\, \ldots  \,n\,1)$ (in cycle notation). The corresponding algebra is {isospectral} and as a corollary of the syzygy relation \cite{KrTk18a}, the spectrum of each idempotent is the set of roots of $z^n-1$.
\end{remark}


Below we consider the general case of a permutation $\sigma$ with several  cycles. It turns out that the resulting algebra decomposes as a direct sum of the ideals corresponding to
the decomposition of $\sigma$ into disjoint permutations.
By abuse of notation, we shall write
\begin{equation}\label{vsig1}
\Sspan{\sigma_j}=\Span{\{e_i:i\in [\sigma_j]\}}=\bigoplus_{i\in [\sigma_j]} \Sspan{e_i},
\end{equation}
where $[\sigma_j]\subset \bar n=\{1,2,\ldots,n\}$ is the orbit of the cycle $\sigma_j$. Each cycle $\sigma_j$ acts as a cyclic permutation of order $|\sigma_j|$ on the orbit $[\sigma_j]$ such that the set of indices $\bar{n}$ is a disjoint union of the orbits $[\sigma_j]$, $1\le j\le r$.


\begin{proposition}\label{prB}
Let a permutation $\sigma\in S_n$ have a disjoint cycle decomposition $\sigma=\sigma_1\ldots \sigma_r$. Then
\begin{equation}\label{desired}
(\Field^{n} ,\bullet_\sigma)\cong\bigoplus_{j=1}^r (\Sspan{\sigma_j},\bullet_\sigma), \qquad (\Sspan{\sigma_j},\bullet_\sigma)\cong (\Field^{ |\sigma_j|} ,\bullet_{\sigma_j}).
\end{equation}
\end{proposition}

\begin{proof}
In the above notation, we have
\begin{equation}\label{vsig10}
(\Sspan{\sigma_j},\bullet)\cong (\Field^{ |\sigma_j|},\bullet).
\end{equation}
The corresponding isomorphism $\psi_{\sigma}$ of $(\Field^{n} ,\bullet)$ in \eqref{psiiso} decomposes in the direct sum $\psi_{\sigma}=\bigoplus_{1\le j\le r} \psi_{\sigma_j}$, where each $\psi_{\sigma_j}\in \End(\Sspan{\sigma_j})$ is determined as the restriction of $\psi_\sigma$ to $\Sspan{\sigma_j}$ and moreover
\begin{equation}\label{vsig2}
(\Field^{n},\bullet)= \bigoplus_{j=1}^r (\Sspan{\sigma_j},\bullet).
\end{equation}
It follows from the definitions that $\Sspan{\sigma_i}\bullet \Alg\subset \Sspan{\sigma_i}$ and $\Sspan{\sigma_i}\bullet \Sspan{\sigma_j}=0$ for $i\ne j$, therefore \eqref{vsig2} is a decomposition into the direct sum of \textit{ideals}.
Let $\pi_j:\Field^{n}\to \Sspan{\sigma_j}$ denote the canonical projection (a linear homomorphism).
Then $\pi_j:(\Field^{n},\bullet) \to (\Sspan{\sigma_j}, \bullet)$ is also an algebra homomorphism and the following commutative diagram holds:
\begin{equation}\label{vsig3}
\begin{CD}
(\Field^{n},\bullet) @>\pi_j>> (\Sspan{\sigma_j}, \bullet)\\
@VV \psi_{\sigma}V @VV\psi_{\sigma_j}V\\
(\Field^{n},\bullet) @>\pi_j>>(\Sspan{\sigma_j}, \bullet)
\end{CD}
\end{equation}
 where the vertical arrows are algebra isomorphisms.
Since $\psi_{\sigma_j}\in \Aut(\Sspan{\sigma_j},\bullet)$, arguing as above we conclude that
\begin{equation}\label{vsig20}
(\Field^{n},\bullet_\sigma)= \bigoplus_{j=1}^r (\Sspan{\sigma_j},\bullet_\sigma),
\end{equation}
where $(\Sspan{\sigma_j},\bullet_\sigma)\cong (\Field^{ |\sigma_j|},\bullet_{\sigma_j})$ are pairwise orthogonal ideals in $(\Field^{n},\bullet_\sigma)$, implying \eqref{desired}.
\end{proof}

By virtue of Proposition~\ref{prB} and Proposition~\ref{pro:decompose}, it suffices to study inner isotopes $(\Field^n,\bullet_\tau)$  for the case of a single cycle $\tau$, we will consider this in the next section.

\section{The  structure of idempotents}\label{sec:single}
First we consider the case $(\Field^n,\bullet_\tau)$ (which is equivalent to the  construction mentioned in Remark~\ref{rem:cyc}) in more detail, i.e. we assume that $\tau$ contains a single cycle. More precisely, let $n\ge 2$ be an integer and $\tau=(2\,3\, \ldots  \,n\,1)\in S_n$ be the right cyclic shift (in cycle notation).
We shall suppose that a field $\Field$ is $\tau$-admissible (see Definition~\ref{def:adm} above), i.e. there are primitive roots of unity of orders $n$ and $2^n-1$ in $\Field$; denote them by $\epsilon$ and $\zeta$, respectively.

\begin{proposition}\label{pro:tau}
The set of nonzero idempotents of $(\Field^n,\bullet_\tau)$ can be  parameterized by
\begin{equation}\label{xcyc2}
\{c_k=(\zeta^{2^{n-1}k},\zeta^{2^{n-2}k},\ldots,\zeta^{2k},\zeta^k):\quad k\in \mathbb{Z}/(2^n-1)\mathbb{Z}\},
\end{equation}
where all $c_k$ are pairwise distinct. The idempotents satisfy
\begin{equation}\label{xcyc3}
c_i\bullet_\tau c_j =c_{i\circledast j},
\end{equation}
where the binary operation $\circledast$ on $\mathbb{Z}/(2^n-1)\mathbb{Z}$ is defined by
\begin{equation}\label{xcyc4}
i\circledast j\equiv \frac12(i+j)\equiv 2^{n-1}(i+j)\mod (2^n-1).
\end{equation}
\end{proposition}

\begin{proof}
The multiplication of $c=(x_1,\ldots,x_n)$ in $(\Field^n,\bullet_\tau)$ is given  by
$$
c\bullet_\tau c=(x_{\tau(1)}^2,\ldots, x_{\tau(n)}^2)=(x_2^2,x_3^2,\ldots,x_n^2,x_1^2),
$$
thus $c$ is an idempotent if and only if
\begin{equation}\label{xcyc}
x_{i+1}^2=x_i\quad \text{ for each $i\in \mathbb{Z}/n\mathbb{Z}$.}
\end{equation}
Iterating the latter relations $n$ times yields $x_{i}^{2^n}=x_i$ for any $i$. Together with \eqref{xcyc} this implies that \textit{either all of $x_i$ are zero (and in that case $c=0$), or all $x_i$ are nonzero, and in the latter case they satisfy}
\begin{equation}
\label{ccyc}
x_{i}^{2^n-1}-1=0, \qquad \forall i\in \mathbb{Z}/n\mathbb{Z}.
\end{equation}
It follows from \eqref{xcyc} that any nonzero idempotent $x$ can be written as
\begin{equation}\label{xcyc1}
x=(t^{2^{n-1}},t^{2^{n-2}},\ldots,t^{2^{1}},t),
\end{equation}
where $t$ is a primitive root of unity of order $2^n-1$, hence   \eqref{xcyc1} readily implies  \eqref{xcyc2}.
Finally, we have
$$c_i\bullet_\tau c_j =(\zeta^{2^{n-2}(i+j)},\zeta^{2^{n-3}(i+j)},\ldots,\zeta^{(i+j)},\zeta^{2^{n-1}(i+j)})
$$
implying \eqref{xcyc3}.

\end{proof}

Identity \eqref{xcyc3} expresses the fact that the product of any two nonzero idempotents in $\Idm(\Field^n, \bullet_\tau)$ is a nonzero idempotent again, in other words, the set $\Idm(\Field^n, \bullet_\tau)$ is a multiplicative magma. Furthermore,  \eqref{xcyc3} implies a magma isomorphism
$$
\Idm(\Field^n, \bullet_\tau)\cong (\mathbb{Z}/(2^n-1)\mathbb{Z},\,\circledast)
$$
Moreover, we have

\begin{proposition}\label{promagma}
The multiplicative magma $(\mathbb{Z}/(2^n-1)\mathbb{Z},\,\circledast)$ is  a commutative idempotent medial quasigroup.
\end{proposition}

\begin{proof}
The quasigroup property can be seen as follows: if $s,t,r\in (\mathbb{Z}/(2^n-1)\mathbb{Z},\,\circledast)$ are such that $s \circledast r=s \circledast t$ then $(t-r)/2\equiv 0\mod (2^n-1)$, hence $t=r$ in $(\mathbb{Z}/(2^n-1)\mathbb{Z},\,\circledast)$. Also, given arbitrary $s,t\in (\mathbb{Z}/(2^n-1)\mathbb{Z},\,\circledast)$, there exists precisely one solution $r:=2t-s$ to the following equation:
$$
s \circledast r=r\circledast s=\frac12(s+2t-s)=t.
$$
Next, $s\circledast s=s$ for all $s\in (\mathbb{Z}/(2^n-1)\mathbb{Z},\,\circledast)$, thus the quasigroup is idempotent. Finally, since
$$
(i \circledast j) \circledast (k \circledast l)\equiv \frac12(i+j+k+l)\mod (2^n-1)
$$
is totally symmetric in all variables, $(\mathbb{Z}/(2^n-1)\mathbb{Z},\,\circledast)$ is a medial quasigroup.
\end{proof}


\begin{proposition}\label{pro:eigentau}
The algebra $(\Field^n,\bullet_\tau)$ has exactly $2^n-1$ distinct regular idempotents $c_k$, i.e. it is generic. Each idempotent $c_k\in \Idm(\Field^n,\bullet_\tau)$ has  the spectrum $\epsilon,\epsilon^2,\ldots, \epsilon^n $, each eigenvalue has multiplicity one. In other words, the characteristic polynomial of $L_{\bullet_\tau}(c_k)$ is given by
\begin{equation}\label{char2}
\det(\lambda\mathds{1}-L_{\bullet_\tau}(c_k))=\lambda^n-1.
\end{equation}
In particular,
\begin{equation}
\label{Ln}
(L_{\bullet_\tau}(c_k))^n=\mathds{1}.
\end{equation}
\end{proposition}

\begin{proof}
Given $p\in \mathbb{Z}/n\mathbb{Z}$ and $k\in \mathbb{Z}/(2^n-1)\mathbb{Z}$, we define
\begin{equation}\label{eigenp}
\eta_{k,p}:=(z_1, \epsilon^p z_2, \, \epsilon^{2p} z_3,\ldots , \epsilon^{(n-1)p}z_n),
\end{equation}
where $z_i$ will be specified later.
By  \eqref{xcyc2} to  $c_k=(\zeta^{2^{n-1}k},\zeta^{2^{n-2}k},\ldots,\zeta^{2^{1}k},\zeta^k)$, hence
\begin{align*}
c_k\bullet_\tau \eta_{k,p}&=
(
\epsilon^p \zeta^{2^{n-2}k}z_2,\,
\epsilon^{2p} \zeta^{2^{n-3}k}z_3,\,
\ldots\,
\epsilon^{(n-1)p} \zeta^{2^{0}k}z_n,\,
\epsilon^{np} \zeta^{2^{n-1}k}z_1
)\\
&=\epsilon^p\cdot (
 \zeta^{2^{n-2}k}z_2,\,
\zeta^{2^{n-3}k}\epsilon^{p} z_3,\,
\ldots\,
\zeta^{2^{0}k}\epsilon^{(n-2)p} z_n,\,
\zeta^{2^{n-1}k}\epsilon^{(n-1)p} z_1
),
\end{align*}
therefore, setting
$$
z_i:=\zeta^{-(2^{n-2}+\ldots +2^{n-i})k}z_1\quad \text{ for $i=2,3,\ldots, n$}
$$
we see that
$$
\zeta^{2^{n-i}k}z_i=\zeta^{2^{n-i}k}\cdot\zeta^{-(2^{n-2}+\ldots +2^{n-i})k}z_1=z_{i-1},\quad 2\le i\le n.
$$
Since $2^{n-2}+\ldots +2^{1}+2^0=2^{n-1}-1\equiv -2^{n-1}\mod (2^n-1)$, we get
$$
z_n=\zeta^{-(2^{n-2}+\ldots +2^{1}+2^0)k}z_1=\zeta^{2^{n-1}k}z_1
$$
implying together with the above that
\begin{equation}\label{eigen0}
c_k\bullet_\tau \eta_{k,p}=
\epsilon^p\cdot (z_1, \epsilon^p z_2, \, \epsilon^{2p} z_3,\ldots , \epsilon^{(n-1)p}z_n)=\epsilon^p \eta_{k,p}.
\end{equation}
In other words,
\begin{equation}\label{eigen1}
\eta_{k,p}=(1,\epsilon^p \zeta^{-2^{n-2}k}, \, \epsilon^{2p} \zeta^{-(2^{n-2}+2^{n-2})k},\ldots , \epsilon^{(n-1)p}\zeta^{-(2^{n-2}+\ldots +2^{1}+2^0)k})
\end{equation}
is an eigenvector of $L_{\bullet_\tau}(c_k)$ with eigenvalue $\epsilon^p$, for any $p\in \mathbb{Z}/n\mathbb{Z}$. Since all $\epsilon,\epsilon^2,\ldots, \epsilon^n $ are pairwise distinct, for the dimension reasons this implies that each eigenvalue $\epsilon^p$ is simple, and moreover the  eigen-decomposition of $L_{\bullet_\tau}(c_k)$ holds:
\begin{equation}\label{eigentau}
(\Field^n, \bullet_\tau)=\bigoplus_{p=1}^n \Span{\eta_{k,p}}.
\end{equation}
This also implies the explicit form of the characteristic polynomial of $L_{\bullet_\tau}(c_k)$ is given by \eqref{char2}. This implies that $\det(\frac12\mathds{1}-L_{\bullet_\tau}(c_k))\ne0$, therefore $(\Field^n, \bullet_\tau)$ is a generic algebra.
\end{proof}

\begin{remark}
It follows from \eqref{char2} that the generic algebra $(\Field^{n}, \bullet_\tau)$ is \textit{isospectral}. On the other hand, there are nongeneric commutative isospectral algebras  containing \textit{infinitely} many idempotents. Then their Peirce spectrum (i.e. the total set of eigenvalues of all idempotents) can have a completely different structure. This holds for the Hsiang algebras that appear in the context of cubic minimal cones; we refer an interested reader to \cite{NTVbook}, \cite{Tk18e}, \cite{Tk19a} for more details.
\end{remark}

We need the following auxiliary property

\begin{lemma}\label{lemzero}
$\sum_{k=1}^{2^n-1}(L_{\bullet_\tau}(c_k))^s=0$ for any $s\in \{1,2,\ldots, n-1\}$.
\end{lemma}

\begin{proof}
Note that by the definition of isotopy, $x\bullet_\tau y=\tau(x\bullet y)$, where $\tau(x_1,\ldots,x_{n-1},x_n)=(x_2,\ldots,x_n,x_1)$ and $\bullet$ is the (commutative associative) coor\-di\-nate-wise  multiplication on $\Field^n$. Write \eqref{xcyc2} as $c_k=(a_1^k,\ldots,a_n^k)$, where $a_i=\zeta^{2^{n-i}}$ do not depend on $k$. Then iterating the definition of $\bullet_\tau$ we obtain
\begin{align*}
(L_{\bullet_\tau}(c_k))^s x&=\tau^s(c_k)\bullet \tau^{s-1}(c_k)\bullet\ldots \bullet \tau(c_k)\bullet  \tau^s(x)\\
&=\omega_k\bullet  \tau^s(x),
\end{align*}
where $\omega_k=\tau^s(c_k)\bullet \tau^{s-1}(c_k)\bullet\ldots \bullet \tau(c_k)$. Since $\tau$, \ldots, $\tau^s$ are cyclic coordinate shifts $\omega_k=(\zeta^{km_1},\zeta^{km_2},\ldots, \zeta^{km_n})$, where $m_i\in \mathbb{Z}_{2^n-1}$ do not depend on $k$, more precisely
$$
m_j=2^{n-j-1}+\ldots+2^{n-j-s}=2^{n-j-s}\cdot(2^{s}-1)\ne0 \quad\text{in $\mathbb{Z}_{2^n-1}$}
$$
Therefore $\zeta^{km_j}\ne1$ and we have
\begin{align*}
\sum_{k=1}^{2^n-1}(L_{\bullet_\tau}(c_k))^s x&=
\sum_{k=1}^{2^n-1} \omega_k\bullet  \tau^s(x)\\
&=\biggl(\sum_{k=1}^{2^n-1} (\zeta^{km_1},\zeta^{km_2},\ldots, \zeta^{km_n})\biggr)\bullet  \tau^s(x)=0
\end{align*}
because $\sum_{k=1}^{2^n-1}\zeta^{km_j}=\zeta^{m_j}(\zeta^{m_j(2^n-1)}-1)(\zeta^{m_j}-1)^{-1}=0$ for any $m_j$, the claim follows.
\end{proof}

\begin{corollary}\label{coraxial}
$\Span{\Idm(\Field^n, \bullet_\tau)}=\Field^n$.
\end{corollary}

\begin{proof}
Let $x\in \Field^n$ and $z:=\sum_{s=0}^{n-1}(L_{\bullet_\tau}(c_k))^s x$. Then using \eqref{Ln} we find
$$
L_{\bullet_\tau}z=\sum_{s=1}^{n}(L_{\bullet_\tau}(c_k))^s x=\sum_{s=0}^{n-1}(L_{\bullet_\tau}(c_k))^s x=z,
$$
therefore $z$ is an eigenvector of $L_{\bullet_\tau}(c_k)$ with eigenvalue $1$, therefore by Proposition~\ref{pro:eigentau}, $z\in \Span{c_{k}}$. This yields $\sum_{s=0}^{n-1}(L_{\bullet_\tau}(c_k))^s x=\mu_kc_k$ for some $\mu_k\in \Field$. Summing up the obtained identities and applying Lemma~\ref{lemzero}, we get
\begin{align*}
\sum_{k=1}^{2^n-1}\mu_kc_k&=\sum_{k=1}^{2^n-1}\sum_{s=0}^{n-1}(L_{\bullet_\tau}(c_k))^s x\\
&=\sum_{s=0}^{n-1}\sum_{k=1}^{2^n-1}(L_{\bullet_\tau}(c_k))^s x\\
&=\sum_{k=1}^{2^n-1} x+\sum_{s=1}^{n-1}\sum_{k=1}^{2^n-1}(L_{\bullet_\tau}(c_k))^s x\\
&=(2^n-1)x,
\end{align*}
and since $2^n-1\ne0 $ in $\Field$, we arrive at the desired conclusion.
\end{proof}

\begin{theorem}
$\Alg:=(\Field^{n} ,\bullet_\tau)$ is an axial algebra with the cyclic fusion law
\begin{equation}\label{etarule}
\Alg_{\epsilon^{p}}(c_k)\ast \Alg_{\epsilon^{q}}(c_k)= \Alg_{\epsilon^{p+q}}(c_k), \quad \forall p,q\in \mathbb{Z}_n.
\end{equation}
\end{theorem}

\begin{proof}
By Corollary~\ref{coraxial}, $\Alg$ is spanned by the set of nonzero idempotents $\Idm(\Alg)$ and by Proposition~\ref{pro:eigentau} all nonzero idempotents in $(\Field^{n} ,\bullet_\tau)$ have the same spectrum $\epsilon,\epsilon^2,\ldots, \epsilon^n $, each eigenvalue $\epsilon^p$ is simple, and the eigen-decomposition \eqref{eigentau} holds. Moreover, applying \eqref{eigen1} we obtain for the corresponding eigenvectors
\begin{align*}
\eta_{k,p}\bullet_\tau \eta_{k,q}&=(\epsilon^{p+q} \zeta^{-2^{m-1}k},\epsilon^{2(p+q)} \zeta^{-(2^{m-1}+2^{m-2})k},\ldots)=\epsilon^{p+q}\zeta^{-2^{m-1}k} \eta_{k,p+q},
\end{align*}
which gives  $\Span{\eta_{k,p}}\bullet_\tau \Span{\eta_{k,q}}=\Span{\eta_{k,p+q}}$ and thereby implies the fusion law \eqref{etarule}.
\end{proof}

Now  we are ready to formulate our main result for the  case of an arbitrary $\sigma\in S_n$. Combining the above results  with Proposition~\ref{prB} and Proposition~\ref{pro:decompose} we arrive at the following general conclusion:

\begin{theorem}\label{thB}
Let a permutation $\sigma\in S_n$ have the disjoint cycle decomposition $\sigma=\sigma_1\ldots \sigma_r$ and a field $\Field$ be $\sigma$-{admissible}. Then the following properties hold:
\begin{itemize}
\item[(a)]
There are exactly $2^n$ distinct regular idempotents in $(\Field^{n} ,\bullet_\sigma)$ naturally stratified in $2^r$ classes $I_\alpha$, enumerated by binary codes $\alpha\in \mathbb{F}_2^{ r}$.
\item[(b)]
For each $\alpha\in \mathbb{F}_2^{ r}$, all idempotents in $I_\alpha$ have the same spectrum. More precisely,
\begin{equation}\label{char5}
\det(\lambda\mathds{1}-L_{\bullet_\tau}(c))=\prod_{i=1}^r(\lambda^{|\sigma_i|}-\alpha(i)), \qquad \forall c\in I_\alpha.
\end{equation}
\item[(c)]
The algebra $(\Field^{n} ,\bullet_\sigma)$ is generic.

\end{itemize}
\end{theorem}


\section{Automorphisms}\label{sec:automorph}
In order to describe the automorphism group for $(\Field^n, \bullet_\sigma)$, we recall some  definitions.
For two groups $G$, $H$, and an action $f: H \to \Aut(G)$, the corresponding \textit{semidirect product} $G\rtimes_f H$ is defined by the group multiplication on $G\times H$ given by $(g_1,h_1)(g_2,h_2)=(g_1f_{h_1}(g_2),h_1h_2).$ In particular, if $H=\Aut(G)$ with $f=\mathrm{id}$ then one obtains the classical notion of the \textit{holomorph} of a group $G$ is the semi-direct product $G\rtimes_\mathrm{id} \Aut(G)$ with the  multiplication given by
\begin{equation}\label{semidirect}
(g_1,\alpha_1)\cdot (g_2, \alpha_2)=( g_1 \alpha_1(g_2), \,\alpha_1 \alpha_2)
\end{equation}
The automorphism group of the additive cyclic group $\mathbb{Z}_N:=(\mathbb{Z}_N,+)$ is isomorphic to the multiplicative group $(\mathbb{Z}_N)^\times=(\mathbb{Z}_N^\times,\cdot)$ of integers modulo $N$ (the group of multiplicative units):
$$
\Aut(\mathbb{Z}_N)\cong (\mathbb{Z}_N)^\times,
$$
where we write for short
$$
\mathbb{Z}_k:=(\mathbb{Z}/k\mathbb{Z},+),\qquad (\mathbb{Z}_N)^\times=(\mathbb{Z}_N^\times,\cdot).
$$
The group $(\mathbb{Z}_N)^\times$ is not cyclic in general, but by the fundamental theorem of finite abelian groups, it is isomorphic to a direct product of cyclic groups of prime power orders. For our analysis the relevant case is when $N=2^n-1$, $n\in \mathbb{Z}^+$.
Then $(\mathbb{Z}_N)^\times$ is the direct product of the groups corresponding to each of the (odd) prime power factors $N=p_1^{k_1} \ldots p_s^{k_s}$:
$$
(\mathbb{Z}_N)^\times=(\mathbb{Z}_{p_1^{k_1}})^\times\times \ldots \times (\mathbb{Z}_{p_s^{k_s}})^\times\cong
C_{p_1^{k_1}-p_1^{k_1-1}}\times \ldots
C_{p_s^{k_s}-p_s^{k_s-1}},
$$
where $C_m$ denote a cyclic group of order $m$. For example, for $n=6$, $B=2^6-1=63=3^2\cdot 7$, hence
$$
(\mathbb{Z}_{63})^\times=C_6\times C_6.
$$

A relevant in the present context is the general affine group of $\mathbb{Z}_N$ which is isomorphic to the holomorph of $\mathbb{Z}_N$:
\begin{equation}\label{affine}
\begin{split}
\mathrm{Aff}(\mathbb{Z}_N)&\cong \mathbb{Z}_N\rtimes_\mathrm{id} (\mathbb{Z}_N)^\times\\
&\cong \{\left(
  \begin{array}{cc}
    m & k \\
    0 & 1 \\
  \end{array}
\right): \,  m\in \mathbb{Z}_N^\times , \,\,k\in \mathbb{Z}_N\}\\
&\cong \{\psi_{m,k}(i)=mi+k: \,  m\in \mathbb{Z}_N^\times , \,\,k\in \mathbb{Z}_N\},
\end{split}
\end{equation}
where the last line is the group of compositions of affine functions $\psi_{m,k}:\mathbb{Z}_N\to \mathbb{Z}_N$,
\begin{equation}\label{composition}
\psi_{m,k}\circ \psi_{m',k'}=\psi_{mm',mk'+k}.
\end{equation}
The exponential map
\begin{equation}\label{deltadef}
\delta(i):=2^i, \quad \delta:(\mathbb{Z}_{n},+)\to (\mathbb{Z}_{2^n-1}^\times,\cdot)\cong \Aut((\mathbb{Z}_{2^n-1},+)),
\end{equation}
is a  well-defined injective group homomorphism, hence there holds the following exact sequence of abelian groups:
\begin{equation}\label{mudef}
0\longmapsto (\mathbb{Z}_{n},+) \stackrel{\delta}{\longmapsto} (\mathbb{Z}_{2^n-1}^\times,\cdot)\stackrel{\mu}{\longmapsto}(\mathbb{Z}_{2^n-1}^\times,\cdot)/(\mathbb{Z}_{n},+)\longmapsto0.
\end{equation}
Denote
$$
\Delta_n=\im \delta=\{1,2,\ldots,2^{n-1}\}\subset \mathbb{Z}_{2^n-1}.
$$
Note that $\Delta_n\cong (\mathbb{Z}_{n},+)\cong C_n$ is a multiplicative subgroup of $(\mathbb{Z}_{2^n-1})^\times$.
We shall also need the semi-direct product
\begin{equation}\label{delta2}
(\mathbb{Z}_{2^n-1},+)\rtimes_\delta (\mathbb{Z}_{n},+)\cong
\{\psi_{2^q,k}(i): \,  q,k\in \mathbb{Z}_n\}
\end{equation}

\begin{remark}\label{rem:aut}
Notice that  any algebra automorphism stabilizes the algebra idempotents. Therefore the automorphism group of an algebra $\Alg$ is a  subgroup of the group of symmetries of nonzero idempotents of $\Alg$. As above, it suffices to consider the case when  $\sigma=\tau\in S_n$ is a single cycle element(the right shift permutation). In this case, Proposition~\ref{promagma} yields that  the set of nonzero idempotents is a multiplicative quasigroup. Below we completely characterize its automorphism group.
\end{remark}

\begin{theorem}\label{th:quasi}
Let $\tau=(2\,3\, \ldots  \,n\,1)\in S_m$ be the right cyclic shift. Then the idempotent quasigroup is
\begin{equation}\label{aut10}
\Aut(\Idm(\Field^{n} ,\bullet_\tau))\cong \mathbb{Z}_{2^n-1}\rtimes_{\mathrm{id}} \mathbb{Z}_{2^n-1}^\times\cong\mathrm{Aff}(\mathbb{Z}_{2^n-1}).
\end{equation}
\end{theorem}

\begin{proof}
Given a pair $m\in \mathbb{Z}_{2^n-1}^\times$ and $k\in \mathbb{Z}_{2^n-1}$, we define a map $\psi_{m,k}(c_i):=c_{mi+k}$, $i\in \mathbb{Z}_{2^n-1}$, of
$\Idm(\Field^{n} ,\bullet_\tau))$ to itself.
Then for any pair $i,j\in \mathbb{Z}_{2^n-1}$ we obtain using \eqref{xcyc3}--\eqref{xcyc4}
\begin{align*}
\psi_{m,k}(c_i\bullet_\tau c_j)&=\psi_{m,k}(c_{{2^{n-1}}(i+j)})=c_{{2^{n-1}}m(i+j)+k}=c_{{2^{n-1}}m(i+j)+{2^{n}}k}\\
&=c_{{2^{n-1}}(mi+k+mj+k)}=c_{mi+k}\bullet_\tau c_{mj+k}\\
&=\psi_{m,k}(c_i)\bullet_\tau \psi_{m,k}(c_j)
\end{align*}
i.e. $\psi_{m,k}\in \Aut(\Idm(\Field^{n} ,\bullet_\tau))$.

In the converse direction, if $g\in \Aut(\Idm(\Field^n ,\bullet_\tau))$ then $g(c_i)=c_{h(i)}$ for some bijection $h:\mathbb{Z}_{2^n-1}\to \mathbb{Z}_{2^n-1}$, hence for any pair $i,j\in \mathbb{Z}_{2^n-1}$
$$
c_{h({2^{n-1}}(i+j))}=g(c_{{2^{n-1}}(i+j)})=g(c_i\bullet_\tau c_j)=g(c_i)\bullet_\tau g(c_j)=c_{{2^{n-1}}(h(i)+h(j))},
$$
implying
$
h({2^{n-1}}(i+j))= {2^{n-1}}h(i)+{2^{n-1}}h(j).
$
Multiplying this  by $2$ and setting $i=j+2$ yields in view of $2^n=1$ in $\mathbb{Z}_{2^n-1}$ that
$$
2h(2^{n}j+2^{n})= 2h(j+1)= h(j+2)+h(j),
$$
hence
\begin{equation}\label{kdef1}
h(j+2)-h(j+1)= h(j+1)-h(j)= \ldots = h(2)-h(1)=:m.
\end{equation}
Since $h$ is an injection, $m\ne 0$ in $\mathbb{Z}_{2^n-1}$, thus $m\in \mathbb{Z}_{2^n-1}^\times$. Therefore \eqref{kdef1} implies $h(j)= m\cdot(j-1)+h(1)= mj+k$, where $k:=h(1)-m\in \mathbb{Z}_{2^n-1}$. This yields   $g=\psi_{m,k}$ and
\eqref{affine} establishes the desired isomorphisms in \eqref{aut10}.

\end{proof}

Now we consider the automorphism group of the ambient algebra $(\Field^n ,\bullet_\tau)$. Recall that by Remark~\ref{rem:aut}, $\Aut(\Field^n ,\bullet_\tau)$ is a subgroup of $\Aut(\Idm(\Field^n ,\bullet_\tau))$. A part of $\Aut(\Field^n ,\bullet_\tau)$ can be identified by the definitions. More precisely, let $k,q\in \mathbb{Z}_n$ and consider the maps given by
\begin{align*}
\alpha_q(x_1,x_2,\ldots,x_n)&=(x_{1-q},x_{2-q},\ldots,x_{n-q})\\
\beta_k(x_1,x_2,\ldots,x_n)&=(\zeta^{2^{n-1}k}x_1,\zeta^{2^{n-2}k}x_2,\ldots,\zeta^{k}x_n)
\end{align*}
where $\alpha_q$ is the right cyclic shift of the coordinates (understood as elements of $\mathbb{Z}/n\mathbb{Z}$), for example, $\alpha_1(x)=(x_n,x_1,x_2,\ldots,x_{n-1})$ etc.

\begin{lemma}\label{lem:auto}
In the notation of \eqref{affine},
\begin{align}
\alpha_q&=\psi_{2^q,0}\label{alpha}\in \Aut(\Field^n ,\bullet_\tau)\\
\beta_k&=\psi_{1,k}\label{beta}\in \Aut(\Field^n ,\bullet_\tau)
\end{align}
Furthermore, $\langle \alpha_1,\beta_1\rangle \cong \mathbb{Z}_{2^n-1}\rtimes_\delta \mathbb{Z}_{n}$, where $\delta$ is defined in \eqref{deltadef}.
\end{lemma}

\begin{proof}
By their definitions, both $\alpha_q$ and $\beta_k$ are linear isomorphisms of $\Field^n$. An easy verification implies the relation $\alpha_q(x\bullet_\tau y)=\alpha_q(x)\bullet_\tau \alpha(y)$. Furthermore,
\begin{align*}
\beta_k(x\bullet_\tau y)&=\beta((x_2y_2,\ldots, x_ny_n,x_1y_1))\\
&=(\zeta^{2^{n-1}k}x_2y_2,\ldots, \zeta^{2^{1}k}x_ny_n,\,\zeta^{2^{0}k}x_1y_1),
\end{align*}
and on the other hand,
\begin{align*}
\beta_k(x)\bullet_\tau \beta_k(y)&=(\zeta^{2^{n-2}k}x_2\cdot \zeta^{2^{n-2}k}y_2,\ldots,
\,\zeta^{2^{n-1}m}x_1\cdot \zeta^{2^{n-1}m}y_1).
\end{align*}
Comparing the obtained expressions yields $\beta_k(x\bullet_\tau y)=\beta_k(x)\bullet_\tau \beta_k(y)$, hence both $\alpha_q$ and $\beta_k$ are automorphisms of $(\Field^n ,\bullet_\tau)$. Applying the definitions to \eqref{xcyc2} implies \eqref{alpha} and \eqref{beta}. Furthermore, elements $ \alpha_1,\beta_1$ generate a subgroup in $\mathrm{Aff}(\mathbb{Z}_{2^n-1})$ consisting of all elements of the kind $\psi_{2^q,k}$, $q,k\in \mathbb{Z}_n$, which by virtue of \eqref{delta2} implies the last claim of the lemma.
\end{proof}

Recall that the cyclotomic polynomial $\Phi_{N}(z)$   is the unique irreducible polynomial with integer coefficients that is a divisor of $z^N-1$ and is not a divisor of
$z^{k}-1$ for any $k < N$. Its roots are all $N$th primitive roots of unity. It is well known that
\begin{equation}\label{prodphi}
\prod_{d|N}\Phi_d(z)=z^N-1.
\end{equation}

\begin{definition}
Let $n\ge2$ be an integer and
$$\Lambda_n(z):=z+z^2+z^4+\ldots +z^{2^{n-1}}.
$$
The number $n$ is said to be \textit{regular} if the cyclotomic polynomial $\Phi_{2^n-1}(z)$ does \textit{not} divide $\Lambda_n(z^{m})-\Lambda_n(z)$ for all  $m\in \mathbb{Z}_{2^n-1}^\times\setminus \Delta_n$.
\end{definition}

\begin{theorem}\label{the:auto}
If $n\ge2$ is a regular integer then $\Aut(\Field^n ,\bullet_\tau)\cong\mathbb{Z}_{2^n-1} \rtimes_{\mathrm{\delta}} \mathbb{Z}_n.$
\end{theorem}

\begin{proof}
An algebra automorphism stabilizes the set of nonzero idempotents,  inducing an  automorphism on the idempotent quasigroup. By Corollary~\ref{coraxial}, $\Span{\Idm(\Field^n, \bullet_\tau)}=\Field^n$, hence if some $f\in \Aut(\Field^n ,\bullet_\tau)$ stabilize each nonzero idempotent in $\Idm(\Field^n, \bullet_\tau)$ then $f=\mathds{1}$. This implies that
$\Aut(\Field^n ,\bullet_\tau)$ is a subgroup of $\Aut(\Idm(\Field^n ,\bullet_\tau)),$
in particular, any algebra automorphism has the form $\psi_{m,k}$, where $ m\in \mathbb{Z}_{2^n-1}^\times$, $k\in \mathbb{Z}_{2^n-1}$.

Therefore we need to identify only those $\psi_{k,m}\in \mathrm{Aff}(\mathbb{Z}_{2^n-1})$ which can be extended to a linear isomorphism of $\Field^n $. To this end, we note that since $m\in \mathbb{Z}_{2^n-1}^\times$, then $s:=m^{-1}k$ is well-defined and by Lemma~\ref{lem:auto} $\psi_{1,s}\in \Aut(\Field^n)$, therefore using \eqref{composition} we conclude that $\psi_{m,k}\circ \psi_{1,s}=\psi_{m,0}\in \Aut(\Field^n ,\bullet_\tau)$. In other words, we can assume without loss of generality that $k=0$.

So let us assume that $\psi_{m,0}\in \Aut(\Field^n ,\bullet_\tau)$.
By Lemma~\ref{lem:auto}, it suffices to  show that  $m\in \Delta_n=\{1,2,\ldots, 2^{n-1}\}$. We argue by contradiction and assume that $m\in \mathbb{Z}_{2^n-1}^\times\setminus \Delta_n$. Then by virtue of \eqref{xcyc2} we find $c_{2^n-1}=(1,1,\ldots,1)$
and also
\begin{equation}\label{Sm}
H_l:=c_l+c_{2l}+c_{4l}+\ldots+c_{2^{n-1}l}=\Lambda_n(\zeta^l)c_{2^n-1}, \qquad \forall l\in \mathbb{Z}_{2^n-1}.
\end{equation}
Since $\psi_{m,0}(c_i)=c_{mi}$ for any $i\in \mathbb{Z}_{2^n-1}$ (in particular,  $\psi_{m,0}(c_{2^n-1})=c_{2^n-1}$) we have $\psi_{m,0}(H_l)=H_{ml},$ implying by virtue of \eqref{Sm} that
$$
\Lambda_n(\zeta^l)=\Lambda_n(\zeta^{lm}), \quad \forall l\in \mathbb{Z}_{2^n-1}.
$$
Since the latter holds for any primitive root $\zeta$ of unity of order $2^n-1$, we conclude that by the definition,
the cyclotomic polynomial $\Phi_{2^n-1}(z)$ divides $\Lambda_n(z^{lm})-\Lambda_n(z^l)$ for any $l\in \mathbb{Z}_{2^n-1}$, in particular, for $l=1$, which implies that $n$ is not a regular integer, a contradiction.
\end{proof}

\begin{remark}
Conjecturally, all positive integer numbers are regular,  an application of Galois theory of cyclotomic polynomials would be helpful to establish this conjecture, but we now are not able to prove this conjecture in the full generality. There are however several particular cases when the verification can be easily done, for example, for any  Mersenne prime $2^n-1$, $n$ is a regular number. In practice,  a verification of that a given $n$ is regular can be done using the resultant, as it shown in the example below.
\end{remark}

\begin{example}\label{ex:n2}
Let $n=2$, then $\mathbb{Z}_{3}^\times\setminus \Delta_2=\emptyset$, hence $n=2$ is a regular integer, implying that
$$
\Aut(\Field^2 ,\bullet_\tau)\cong\mathbb{Z}_{3} \rtimes_{\mathrm{\delta}} \mathbb{Z}_2\cong S_{3}.
$$
The latter is the well-known fact that $S_3$ is an internal semi-direct product of the subgroups $\mathbb{Z}_3$ and $\mathbb{Z}_2$, where $\mathbb{Z}_3$ is the subgroup generated by one of the two 3-cycles and $\mathbb{Z}_2\cong C_2$ is the subgroup generated by any transposition.
\end{example}

\begin{example}\label{ex:n3}
Let $n=3$, then $\Lambda_3(z)=z+z^2+z^4$ and $\mathbb{Z}_{7}^\times\setminus \Delta_3=\{3,5,6\}$. An easy verification shows that the resultant
$$
R(\frac{\Lambda_3(z^q)-\Lambda_3(z)}{z(z-1)}, \,\Phi_{7}(z))=7^2\ne0, \qquad \forall q\in \{3,5,6\},
$$
hence $\Lambda_3(z^q)-\Lambda_3(z)$ does not have common divisors  with $\Phi_{7}(z)$ (note that $z=0,1$ cannot be common zeros). Therefore $n=3$ is a regular integer. Similarly, for   $\Lambda_4(z)=z+z^2+z^4+z^8$, $\mathbb{Z}_{15}^\times\setminus \Delta_3=\{7,11,13,14\}$ and
$$
R(\frac{\Lambda_4(z^q)-\Lambda_4(z)}{z(z-1)}, \,\Phi_{15}(z))=3^4\cdot 5^4\ne0, \qquad \forall q\in \Delta_3,
$$
\end{example}

Finally, we point out the following useful observation. Let $n\ge 2$ be an integer, and $\Phi_n(z)$ be the cyclotomic polynomial of degree $n$. Consider the quotient polynomial algebra
$$
\mathbb{T}_n:=(\Field[z]/\Phi_{2^n-1}(z),\bullet)\cong (\Field^{\phi(2^n-1)},\bullet)
$$
where $\phi$ is Euler's totient function. Then $\Aut(\mathbb{T}_n)\cong S_{\phi(2^n-1)}$, where any automorphism is a substitution $P(z)\to P(t_\alpha(z))$, with $\alpha\in S_n$ and $t_\alpha(z)$ is the  Lagrange polynomial of degree $\le \phi(2^n-1)$ uniquely determined the relations $t_\alpha(\zeta^k)=\zeta^{\alpha(k)}$ for any $k\in \mathbb{Z}_{2^n-1}^\times$, where $\zeta$ is a fixed primitive root of unity of degree $2^n-1$. Note that $h_m(z):=z^m\in \Aut(\mathbb{T}_n)$, where $m\in \mathbb{Z}_{2^n-1}^\times$ ($\mathbb{Z}_{2^n-1}^\times$ is an abelian subgroup of $S_{\phi(2^n-1)}$).

Now, let us consider $\Lambda_n(z)$ as an element in $\mathbb{T}_n$. By \eqref{prodphi}, $\Phi_{2^n-1}(z)|(z^{2^n-1}-1)$, hence
$$
\Lambda_n(z^2)-\Lambda_n(z)=z^{2^n}-z\equiv 0\mod \Phi_{2^n-1}(z),
$$
therefore $\Lambda_n(z^2)=\Lambda_n(z)$, in other words we conclude that

\begin{proposition}\label{pro:Lam}
$\Lambda_n$ is a fixed point of the natural action of $\Delta_n$ by substitutions. Moreover, the integer $n$ is regular if and only if the stabilizer subgroup of $\Lambda_n$ in $\mathbb{Z}_{2^n-1}^\times$ is exactly $\Delta_n$.
\end{proposition}

\section{Three examples for $n=3$}\label{sec:example}

The main goal of this section is to illustrate our results for $(\Field^{n} ,\bullet_\sigma)$ in the particular case $n=3$. We shall assume that $\Field$ is  splitting field of polynomial $P(z)=z^3-1$ and $\epsilon$ will denote a primitive root of unity of degree 3 (in section~\ref{subsec:3} we additionally assume that also a primitive root of unity of degree 7 exists).
By Theorem~\ref{thA},   there are exactly three distinct (isomorphy classes of) inner isotopes coded by the conjugacy classes of $S_3$, which are in a one-to-one correspondence with integer partitions of $3$, i.e.
\begin{equation}\label{part1}
\begin{split}
3&=1+1+1\\
&=2+1\\
&=3.
\end{split}
\end{equation}
Each of the  three-dimensional isotope algebras  will be considered below.

\subsection{The case $"1+1+1"$: a unital commutative associative algebra}\label{sec:single1}
In this case we have the trivial three cycle partition
$1+1+1$
which uniquely determines the unity in $S_3$: $e=(1)(2)(3)$ (in the cyclic notation), thus
$$
(\Field^3 ,\bullet_e)\cong (\Field^3 ,\bullet)
$$
i.e. the corresponding inner isotope is the associative (direct product) algebra $(\Field^3 ,\bullet)$ itself. The multiplication  structure is a uniquely determined by the multiplication table in the standard basis $\{e_1,e_2,e_3\}$ \eqref{standard}:
\begin{equation*}
e_i\bullet e_j=e_{i+j (\mathrm{mod}\,3)}.
\end{equation*}
The automorphism group is given by  Theorem~\ref{thA}
$$
\Aut(\Field^3 ,\bullet_e)\cong S_3.
$$
The algebra $(\Field^3 ,\bullet_e)$ is generic  and  any algebra idempotent can be written as $c_k=(\alpha_1,\alpha_2,\alpha_3)$, where $(\alpha_1,\alpha_2,\alpha_3)\in \mathbb{F}_2^{ 3}$ is the binary decomposition of $k$, $0\le k\le 7$.
For example, the binary codes $1=001_2$, $2=010_2$ and $4=100_2$ correspond to the three standard basis idempotents $e_3,e_2,e_1$. The idempotent $c_7=e_1+e_2+e_3$ is the algebra {unity}.
This yields the multiplication table (Table~\ref{tab2} below) and the characteristic polynomials $\chi_i$ of $L_{\bullet_e}(c_i)$ are given  by (cf. with \eqref{char2})
\begin{equation}\label{peirce2}
\begin{split}
\chi_0&=\lambda^3,\qquad
\chi_1=\chi_2=\chi_4=(\lambda-1)\lambda^2,\\
\chi_3=\chi_4=\chi_6&=(\lambda-1)^2\lambda,\qquad
\chi_7=(\lambda-1)^3.
\end{split}
\end{equation}

\begin{table}[ht]
\begin{center}
\begin{tabular}{c||ccccccc}
$\bullet_\tau$&1&2&3&4&5&6&7\\\hline\hline
1&1&0&1&0&1&0&1\\
2&0&2&2&0&0&2&2\\
3&1&2&3&0&1&2&3\\
4&0&0&0&4&4&4&4\\
5&1&0&1&4&5&4&5\\
6&0&2&2&4&4&6&6\\
7&1&2&3&4&5&6&7\\
\end{tabular}
\end{center}
\caption{The multiplication table of idempotents $i\sim c_i$ in  $(\Field^3 ,\bullet_e)$ with $e=(1)(2)(3)\in S_3$}\label{tab2}
\end{table}

Note that $(\Field^3 ,\bullet_e)$, as an associative algebra is axial in the sense of Definition~\ref{def:ax} above, it follows from the classical results due to Benjamin Peirce. Indeed, $(\Field^3 ,\bullet_e)=\Span{\{e_1,e_2,e_4\}}$, these idempotents are primitive and satisfye
 the same fusion law:
\begin{equation}\label{fusion0}
\begin{array}{r||rr}
 \bullet_e &\Alg_1& \Alg_0\\\hline\hline
\Alg_1&\Alg_1&\Alg_0\\
 \Alg_0&0 & \Alg_0 \\
\end{array}
\end{equation}
However, in contrast to the single cycle case (see section~\ref{subsec:3} below), the set of nonzero idempotents is \textit{not} a multiplicative magma, see table~\ref{tab2}.


\subsection{The single cycle case $"3"$: a commutative isospectral medial algebra}\label{subsec:3}
The one-cycle partition corresponds to the conjugacy class of $\tau=(2\,3\,1)\in S_3$, i.e.  shifts.  Such algebras have been introduced and studied first for $n=3$  in \cite{KrTk18a} in the context of isospectral algebras and later for any $n\ge2$ in \cite{KrTk23a} in the polynomial setting. It is natural to assume that the ground field $\Field$ additionally contains a primitive root of unity of order $7=2^3-1$, denote it by $\zeta$.

The standard basis elements are no longer idempotents, for example $e_1\bullet_\tau e_1=e_2$. By Proposition~\ref{pro:tau}, the algebra $(\Field^3 ,\bullet_\tau)$ is generic and contains 7 distinct nonzero idempotents, which can be explicitly written by
\begin{equation}\label{xcyc20}
\Idm(\Field^3 ,\bullet_\tau)=\{c_k=(\zeta^{4k},\zeta^{2k},\zeta^k):\quad k\in \mathbb{Z}/7\mathbb{Z}\}.
\end{equation}
The multiplication rule \eqref{xcyc3} between idempotents takes the form
\begin{equation}\label{xcyc30}
c_i\bullet_\tau c_j =c_{4(i+j)}=c_{i\circledast j},
\end{equation}
where $\circledast$ on $\mathbb{Z}/7\mathbb{Z}$ is defined by
\begin{equation}\label{xcyc40}
i\circledast j\equiv 4(i+j)\mod 7.
\end{equation}
The characteristic polynomials are given by
\begin{equation}\label{lamlam}
\det(\lambda\mathds{1}-L_{\bullet_\tau}(c_k))=\lambda^3-1.
\end{equation}
Therefore  $(\Field^3 ,\bullet_\tau)$ is \textbf{isospectral} and furthermore it is an \textbf{axial} algebra with the following {fusion law}:
\begin{equation}\label{fusion2}
\begin{array}{c||ccc}
 \bullet_\omega &\Alg_1& \Alg_{\epsilon} &\Alg_{\epsilon^2}\\\hline\hline
\Alg_1 &\Alg_1& \Alg_{\epsilon} &\Alg_{\epsilon^2}\\
\Alg_{\epsilon} &\Alg_{\epsilon}& \Alg_{\epsilon^2} &\Alg_1\\
\Alg_{\epsilon^2} &\Alg_{\epsilon^2}&\Alg_1 &\Alg_{\epsilon}
\end{array}
\end{equation}

Combining Theorem~\ref{th:quasi} and Theorem~\ref{the:auto} with Example~\ref{ex:n3},  we obtain

\begin{theorem}\label{th:n=3}
The idempotent quasigroup and the algebra automorphism groups are respectively:
\begin{align}
\Aut(\Idm(\Field^3 ,\bullet_\tau))&\cong \mathbb{Z}_7\rtimes_{\mathrm{id}} \mathbb{Z}_7^\times
\label{aut1}\\
\Aut(\Field^3 ,\bullet_\tau)&\cong\mathbb{Z}_7 \rtimes_{\mathrm{\delta}} \mathbb{Z}_3,\label{aut2}
\end{align}
where $\mathbb{Z}_7 \rtimes_{\mathrm{\delta}} \mathbb{Z}_3$ is the smallest  non-abelian group of odd order.
\end{theorem}


Furthermore, $(\Field^3 ,\bullet_\tau)$ has many other remarkable properties (see \cite{KrTk23a} for a more detailed discussion), for example it satisfies  the algebra identity
$$
(x\bullet_\tau (x\bullet_\tau(x\bullet_\tau y)))=\Delta (x)y, \qquad \forall x,y \in (\Field^3 ,\bullet_\tau),
$$
where $\Delta$ is a multiplicative homomorphism of degree $3$ given explicitly  by a circulant:
$$
\Delta(a_0 e_0+ a_1e_1+ a_2e_2)=\left|
  \begin{array}{ccc}
    a_0 & a_2 & a_1 \\
    a_1 & a_0 & a_2 \\
    a_2 & a_1 & a_0 \\
  \end{array}
\right|:(\Field^3 ,\bullet_\tau)\to (\Field,\bullet).
$$

\subsection{The case $[2+1]$}

Finally, we consider the two-cycle partition $[2+1]$
corresponding to the conjugacy class of {transpositions}.  Without loss of generality we can assume that $\omega=(2\,1)(3)\in S_3$. As in section~\ref{sec:single1}, the algebra $(\Field^3 ,\bullet_\omega)$ is decomposable, more precisely:
\begin{equation}\label{ff3}
(\Field^3 ,\bullet_\omega)\cong (\Field^2 ,\bullet_\tau)\times (\Field ,\bullet_e).
\end{equation}
The second factor is trivial and the first factor is the two-dimensional \textit{Harada algebra} \cite{Harada84}, i.e. a uniquely determined up to isomorphism two-dimensional commutative algebra generated by two distinct idempotents $c_1$ and $c_2$ subject to the condition $c_1\bullet c_2=-c_1-c_2$  (we don't need this characterization later and leave the details to an interested reader).

Combining \eqref{ff3} with Example~\ref{ex:n2},  we get

\begin{proposition}
There holds
$
\Aut(\Field^3 ,\bullet_\omega)\cong S_3.
$
\end{proposition}

\section{Final remarks and questions}

There are at least two natural questions that remained unanswered in this article: 1) Determine the automorphism group of the obtained inner isotopes of commutative associative algebras $(\Field^n,\bullet)$ for the general dimensions, and 2) Which of the constructed  above inner isotopes are axial algebras.  Note that in Section~\ref{sec:example} we completely discussed these two questions in the case $n=3$ (the case $n=2$ is trivial). A further analysis reveals that by the same methods are applicable to $n=4$; also some partial results were mentioned in \cite{KrTk23a}.

Another interesting natural question is how to apply the present methods to general \textit{nonassociative} commutative algebras or at least to  inner isotopes of the 2nd order:
$$
(\Alg,\star)\rightsquigarrow(\Alg,\star_h)\rightsquigarrow(\Alg,{\star_h}_g)\rightsquigarrow\ldots
$$
where $h\in \Aut(\Alg, \star)$, $g\in \Aut(\Alg, \star_h)$ etc?

Finally, we mention that some methods and ideas of the present paper (inner isotopies) can be useful in the case $\Field =\mathbb{C}$ in the study, for example, of algebras of holomorphic functions in the unit disk like Bergman and Bloch spaces of holomorphic functions defined on the open unit disc in the complex plane \cite{Hedenmalm}.

{\small\bibliography{cimart_tkachev}}
\EditInfo{September 1, 2023}{July 26, 2024}{Adam Chapman, Mohamed Elhamdadi and Ivan Kaygorodov}
\end{document}